\documentclass{amsart}

\usepackage{amsmath, amsfonts, amssymb}  
\usepackage{enumerate}
\usepackage{enumitem}
\usepackage[hidelinks]{hyperref}
\usepackage{graphicx}
\usepackage[all]{xy}
\usepackage{wrapfig}
\usepackage{fancyvrb}
\usepackage[mathscr]{euscript}
\usepackage[T1]{fontenc}
\usepackage{listings}
\usepackage{tikz}
\usepackage{amsthm}
\usepackage{dsfont}
\usepackage{tikz-cd}
\usepackage{adjustbox}
\usepackage{centernot}
\usepackage{mathtools}
\usepackage{cite}
\usepackage{bbm}
\usepackage{stmaryrd}
\usetikzlibrary{external}
\usepackage[margin=3cm]{geometry}

\makeatletter
\DeclareRobustCommand{\cev}[1]{
  {\mathpalette\do@cev{#1}}
}
\newcommand{\do@cev}[2]{
  \vbox{\offinterlineskip
    \sbox\z@{$\m@th#1 x$}
    \ialign{##\cr
      \hidewidth\reflectbox{$\m@th#1\vec{}\mkern4mu$}\hidewidth\cr
      \noalign{\kern-\ht\z@}
      $\m@th#1#2$\cr
    }
  }
}
\makeatother

\let\oldtocsection=\tocsection

\let\oldtocsubsection=\tocsubsection

\let\oldtocsubsubsection=\tocsubsubsection

\renewcommand{\tocsection}[2]{\hspace{0em}\oldtocsection{#1}{#2}}
\renewcommand{\tocsubsection}[2]{\hspace{1em}\oldtocsubsection{#1}{#2}}
\renewcommand{\tocsubsubsection}[2]{\hspace{2em}\oldtocsubsubsection{#1}{#2}}

\title{A Betti Geometric Casselman--Shalika Equivalence}
\author{Colton Sandvik}
\date{}

\newcommand{\C}{\mathbb{C}}
\newcommand{\Q}{\mathbb{Q}}
\newcommand{\Z}{\mathbb{Z}}
\newcommand{\F}{\mathbb{F}}

\renewcommand{\L}{\mathcal{L}}
\newcommand{\A}{\mathbb{A}}
\newcommand{\G}{\mathbb{G}}
\renewcommand{\k}{\mathbbm{k}}

\DeclareMathOperator{\dR}{dR}

\newcommand{\Dmod}{D\textnormal{-mod}}
\newcommand{\rh}{\textnormal{rh}}
\newcommand{\hol}{\textnormal{hol}}
\renewcommand{\char}{\textnormal{char}}

\DeclareMathOperator{\Sch}{Sch}

\newcommand{\scrF}{\mathcal{F}}
\newcommand{\scrJ}{\mathcal{J}}
\newcommand{\scrG}{\mathcal{G}}

\newcommand{\scrS}{\mathcal{S}}

\newcommand{\scrK}{\mathcal{K}}

\newcommand{\scrP}{\mathcal{P}}

\newcommand{\bfX}{\mathbf{X}}
\newcommand{\bfXv}{\check{\mathbf{X}}}

\newcommand{\IW}{\mathcal{IW}}

\DeclareMathOperator{\Hom}{Hom}

\DeclareMathOperator{\im}{im}

\DeclareMathOperator{\cons}{c}

\DeclareMathOperator{\Ind}{Ind}
\DeclareMathOperator{\Inv}{Inv}
\DeclareMathOperator{\Ext}{Ext}

\DeclareMathOperator{\Rep}{Rep}

\renewcommand{\mod}[1]{#1\textnormal{-mod}}

\newcommand{\DD}{\mathbb{D}}
\renewcommand{\H}{\mathbb{H}}
\DeclareMathOperator{\IC}{IC}
\DeclareMathOperator{\Av}{Av}

\newcommand{\uk}{\underline{\k}}

\DeclareMathOperator{\Spec}{Spec}
\DeclareMathOperator{\For}{For}
\DeclareMathOperator{\pt}{pt}
\DeclareMathOperator{\id}{id}

\newcommand{\scrQ}{\mathcal{Q}}

\DeclareMathOperator{\Gr}{Gr}

\DeclareMathOperator{\Lie}{Lie}
\newcommand{\ilim}{\lim_{\leftarrow}}

\newcommand{\fr}[1]{\mathfrak{#1}}
\DeclareMathOperator{\op}{op}
\DeclareMathOperator{\fg}{fg}

\DeclareMathOperator{\aff}{aff}
\DeclareMathOperator{\Wh}{Wh}
\DeclareMathOperator{\Kir}{Kir}
\DeclareMathOperator{\Whit}{Whit}
\newcommand{\sstar}{\stackrel{*}{\star}}
\newcommand{\cstar}{\stackrel{!}{\star}}

\newtheorem{lemma}[subsubsection]{Lemma}
\newtheorem{proposition}[subsubsection]{Proposition}
\newtheorem{theorem}[subsubsection]{Theorem}

\theoremstyle{definition}
\newtheorem{definition}[subsubsection]{Definition}
\newtheorem{example}[subsubsection]{Example}
\newtheorem{remark}[subsubsection]{Remark}

\newenvironment{midsecproof}[1]{\vspace{\topsep} \noindent \textit{Proof of #1.}}{$\square$ \vspace{\topsep}}

\AtBeginEnvironment{tikzcd}{\tikzexternaldisable}
\AtEndEnvironment{tikzcd}{\tikzexternalenable}

\begin{document}

        \begin{abstract}
                Whittaker sheaves are ubiquitous in geometric representation theory; however, their definition requires one to restrict the sheaf-theoretic setting to either étale sheaves or $D$-modules.
                Gaitsgory and Lysenko proposed a solution to this problem called the Kirillov model, which is well-defined for many sheaf theories.
                In this paper, we advance the study of the Kirillov model in the setting of Betti sheaves with a particular emphasis on developing a theory of Iwahori--Whittaker sheaves on the affine Grassmannian.
                Using this framework, we prove a Betti geometric Casselman--Shalika equivalence, which relates Iwahori--Whittaker perverse sheaves on the affine Grassmannian with the Satake category.
        \end{abstract}

	\maketitle

    \tableofcontents

    \section{Introduction} 
        \subsection{Geometric Satake}

    Let $\k$ be a commutative noetherian ring of finite global dimension.
    Let $G$ be a complex connected reductive group, and write $\L G$ and $\L^+ G$ for its loop group and positive loop group, respectively.
    We will denote the affine Grassmannian for $G$ by $\Gr = \L G / \L^+ G$.

    The celebrated geometric Satake equivalence of Mirković--Vilonen \cite{MV} gives an equivalence of symmetric monoidal categories,
    \[ P_{\L^+ G} (\Gr, \k) \stackrel{\sim}{\to} \Rep (G_{\k}^{\vee}).\]
    Here $P_{\L^+ G} (\Gr, \k)$ denotes the category of $L^+ G$-equivariant perverse sheaves on $\Gr$, and $\Rep (G_{\k}^{\vee})$ denotes the category of representations of the Langlands dual split reductive group over $\k$ on finitely generated $\k$-modules.
    Under this equivalence, the convolution of spherical perverse sheaves corresponds to the tensor product of $G_{\k}^{\vee}$-modules.
    
    One of the remarkable features of the geometric Satake equivalence is the sheer generality in which it holds. Namely, it makes only minor restrictions on the coefficient ring $\k$.
    The equivalence also makes sense for other sheaf theories on the automorphic side. For example, while we stated the equivalence using Betti sheaves, analogous versions exist for étale and de Rham sheaves. 

    \subsection{The Geometric Casselman--Shalika Equivalence}

    We will momentarily switch our sheaf-theoretic setting to étale sheaves.
    Let $\F$ be an algebraically closed field of positive characteristic, and let $\k$ be either a finite field of characteristic $\ell \neq \char (\F)$, or a finite extension of $\Q_{\ell}$, or the ring of integers of such an extension.
    We now view $G$ as being a connected reductive group over $\F$. Similarly, the loop group, positive loop group, and affine Grassmannian are viewed as $\F$-schemes.

    There is another notable description of the automorphic side of the geometric Satake equivalence provided by Bezrukavnikov--Gaitsgory--Mirković--Riche--Rider \cite{BGMRR}.
    Their description gives an equivalence of categories
    \[P_{\L^+ G} (\Gr, \k) \stackrel{\sim}{\to} P_{\IW} (\Gr, \k)\]
    where $P_{\IW} (\Gr, \k)$ denotes the category of Iwahori--Whittaker perverse sheaves on $\Gr$. 
    We call this the \emph{geometric Casselman--Shalika equivalence}.

    The Iwahori--Whittaker category is  not visibly monoidal unlike spherical perverse sheaves; however it two key advantages over the spherical category. First, the (co-)standard objects have more explicit descriptions as Iwahori--Whittaker sheaves.
    As a result, the highest weight structure on $P_{\IW} (\Gr, \k)$ is quite transparent.
    Second, the realization functor
    \[D^b P_{\IW} (\Gr, \k) \to D_{\IW}^b (\Gr, \k)\]
    is an equivalence of categories. Combining this with the geometric Satake equivalence, one obtains an equivalence of triangulated categories
    \[D_{\IW}^b (\Gr, \k) \cong D^b \Rep (G_{\k}^{\vee}). \] 
    As a result, the derived category of $G_{\k}^{\vee}$-representations admits a geometric realization. This is in contrast with $D_{\L^+ G}^b (\Gr, \k)$ where the spectral side is more complicated (see \cite{BF}).

    We note that a de Rham version of the geometric Casselman--Shalika equivalence \cite{ABBGM} predates the étale version of \cite{BGMRR}.
    The restriction to the de Rham or étale setting has historically been a necessary inconvenience when working with Whittaker models.
    Namely, the starting point of the theory is the existence of a non-trivial multiplicative sheaf on $\A^1$, e.g., the exponential $D$-module or an Artin--Schreier sheaf.
    However, no such sheaf exists in the Betti setting. A solution to this problem has been proposed by Gaitsgory and Lysenko called the \emph{Kirillov model} \cite{GaiWhit, GL}.
    Philosophically, the solution replaces the exponential $D$-module or Artin--Schreier sheaf by the Fourier--Laumon kernel, a certain $\G_m$-equivariant sheaf on $\A^1$.
    This affords a version of the Iwahori--Whittaker sheaves in the Betti setting which does not require any restrictions on the coefficient ring.
    
    We can now state the main theorem of the paper, a Betti version of the geometric Casselman--Shalika equivalence.
    \begin{theorem}\label{thm:1}
        There is an equivalence of categories,
        \[P_{\L^+ G} (\Gr, \k) \stackrel{\sim}{\to} P_{\IW} (\Gr, \k).\]
    \end{theorem}

    The general strategy of proving Theorem \ref{thm:1} is closely modeled on that of \cite{ABBGM}.
    However, complications are added when using the Kirillov model.
    \begin{enumerate}
        \item There is a certain lack of symmetry in the Kirillov model. In fact, there is a $!$- and $*$-version of the Kirillov model, and Verdier duality exchanges them.
        Likewise, the $!$-Kirillov model (resp. $*$-Kirillov model) only behaves well with $!$-pushforwards and $*$-pullbacks (resp. $*$-pushforwards and $!$-pullbacks).
        As a result, more care has to be employed when performing sheaf-theoretic arguments. This problem is primarily formal. We will address it by introducing new sheaf functors which serve as replacements for the ones missing due to the lack of symmetry.
        \item The additional $\G_m$-equivariance makes certain cohomological arguments more complicated. For example, there is no obvious fully faithful forgetful functor from $D_{\IW}^b (\Gr, \k)$ to $D_{\cons}^b (\Gr, \k)$.
        \item The Fourier--Laumon kernel is not a local system on $\A^1$. As a result, even after forgetting $\G_m$-equivariance, one cannot use Artin vanishing theorem style arguments when studying the cohomology of a variety with respect to a Fourier--Laumon kernel. 
    \end{enumerate}

    \begin{remark}
    In the final stages of preparing this paper, the author was made aware of related work of Cass, van den Hove, and Scholbach \cite{CvdHS} that uses the Kirillov model to give a motivic-variant of Theorem \ref{thm:1}.
    More precisely, the categories of Betti perverse sheaves can be replaced by categories of (reduced) mixed Tate motives.
    T. van den Hove explained to the author that Theorem \ref{thm:1} can be formally deduced from \cite[Theorem 1.6]{CvdHS} by taking a Lurie tensor product over Tate motives on $\Spec (\Q)$ with Betti sheaves on a point.
    \end{remark}

    \subsection{Acknowledgements}

    The author thanks Paul Sobaje for discussions that helped inspire this project. The author also thanks Pramod Achar for helpful discussions and comments on an earlier draft.
    Finally, the author is grateful to Thibaud van den Hove for explaining the connections between this work and their recently completed motivic version joint with Robert Cass and Jakob Scholbach \cite{CvdHS}.

    The author was partially supported by NSF Grant No. DMS-2231492.

    \section{Affine Grassmannians and Affine Flag Varieties}
        \subsection{Notation}

    Choose a Borel subgroup $B^{-} \subset G$ and let $T \subset B^{-}$ denote its maximal torus.
    We denote by the Borel subgroup opposite to $B^{-}$ with respect to $T$ by $B^+$, and let $U^{+}$ be its unipotent radical.
    Let $\Phi$ be the root system of $G$, and let $\Phi^+ \subset \Phi$ be the set of positive roots consisting of $T$-weights in $\Lie (U^+)$.
    We also write $\Phi_s \subset \Phi$ for the subset of simple roots.
    Let $\bfX = X^* (T)$ denote the set of weights of $T$, $\bfX^+$ the set of dominant weights (with respect to $\Phi^+$).
    We also consider the coweight lattice $\bfXv = X_* (T)$, the dominant coweights $\bfXv^{+} \subset \bfXv$, and the strictly dominant coweights $\bfXv^{++} \subset \bfXv^{+}$.
    Let
    \[\rho = \frac{1}{2}\sum_{\alpha \in \Phi^+} \alpha \in \Q \otimes_{\Z} \bfX.\]
    For each $\alpha \in \Phi_s$, we fix a pinning $u_{\alpha} : U_{\alpha} \stackrel{\sim}{\to} \G_a$, where $U_{\alpha}$ is the root subgroup of $G$ associated to $\alpha$.

    We will assume that there exists $\zeta \in \bfXv$ such that $\langle \zeta, \alpha \rangle = 1$ for any simple root $\alpha \in \Phi_s$. In particular, $\bfXv^{++} = \zeta + \bfXv^+$.
    The requirement that $\zeta$ exists is relatively weak and completely inconsequential from the perspective of the representation theory of $G$.
    Namely, such a cocharacter always exists if $G$ is semisimple and of adjoint type. We will fix $\zeta$ throughout the paper.

    Let $W_f$ be the (finite) Weyl group of $(G,T)$. Define the affine and extended affine Weyl groups by $W_{\aff} \coloneq W \ltimes \Z \Phi^{\vee}$ and $W \coloneq W_f \ltimes \bfXv$, respectively.
    The affine Weyl group is a Coxeter group whereas the extended affine Weyl group is not. 
    Nonetheless, the length function $\ell$ on $W_{\aff}$ extends to a length function on $W$.
    The key obstruction to $W$ being a Coxeter group is the presence of length 0 elements, denoted by $\Omega$.
    The multiplication induces a group isomorphism $W_{\aff} \rtimes \Omega \stackrel{\sim}{\to} W$.
    Similarly, the Bruhat order on $W_{\aff}$ extends to a partial order on $W$.

    \subsection{Geometry}

    The loop group is the ind-scheme $\L G : \Sch^{\op} \to \textnormal{Set}$ defined by $\L G (X) = X (\C((z)))$.
    We can also consider the positive loop group $\L^+ G : \Sch^{\op} \to \textnormal{Set}$ defined by $\L^+ G (X) = X (\C[[z]])$ which is a group subscheme of $\L G$.
    Let $I^{-} \subset \L^+ G$ denote the Iwahori subgroup associated with $B^{-}$, i.e., the subgroup defined by
    \[I^{-}  = \{g \in \L^+ G \mid g\vert_{z=0} \in B^{-} \}. \]
    We also denote by $I^+$ the opposite Iwahori subgroup of $I^{-}$. The pro-unipotent radical of $I^{-}$ (resp. $I^{+}$) will be denoted by $I_u^+$ (resp. $I_u^-$).
    We can now define the affine Grassmannian as the fppf-sheaf quotient
    \[\Gr \coloneq \L G/\L^+ G,\]
    which is represented by an ind-variety.
    
    Any $\lambda \in \bfXv$ defines a point $z^{\lambda} \in \L T \subset \L G$, and hence a point on the affine Grassmannian $L_{\lambda} \coloneq z^{\lambda} \L^ +G / \L^+ G \in \Gr$.
    The Bruhat decomposition yields a decomposition of $\Gr$ into $\L^+ G$-orbits,
    \[\Gr = \bigsqcup_{\lambda \in \bfXv^{+}} \Gr^{\lambda}\]
    where $\Gr^{\lambda} \coloneq \L^+ G \cdot L_{\lambda}$.
    We set $j_{\lambda} : \Gr^{\lambda} \hookrightarrow \Gr$.
    Each $\Gr^{\lambda}$ is smooth and simply connected, moreover, we have that $\dim \Gr^{\lambda} = \langle 2\rho, \lambda \rangle$ (see \cite{Lu1}).
    The closures of $\L^+ G$-orbits are governed by the partial order on dominant coweights,
    \[\overline{\Gr^{\lambda}} = \bigcup_{\mu \preceq \lambda} \Gr^{\mu}.\]

    For $\lambda \in \bfXv$, we will write $P_{\lambda} \subset G$ for the parabolic subgroup of $G$ containing $B^{-}$ associated with the subset of simple roots $\{ \alpha \in \Phi_s \mid \langle \alpha, \lambda \rangle = 0\}$.
    Note that $P_{\lambda}$ is the stabilizer in $G$ of $L_{\lambda}$, so that we have a canonical isomorphism $G/P_{\lambda} \stackrel{\sim}{\to} G \cdot L_{\lambda}$.
    The map
    \[p_{\lambda} : \Gr^{\lambda} \to G/P_{\lambda}, \qquad x \mapsto \lim_{t \to 0} t \cdot x\]
    where $\G_m$-acts on $\Gr$ by loop rotation is a morphism of algebraic varieties. Moreover, $p_{\lambda}$ realizes $\Gr^{\lambda}$ as an affine bundle over $G/P_{\lambda}$ (cf., \cite[Lemme 2.3]{NP}). 

    For $\lambda \in \bfXv^{+}$ we set
    \[\Gr_{\lambda}^{+} \coloneq I^+ \cdot L_{\lambda} \subset \Gr.\]
    These are locally closed subvarieties and there is a decomposition
    \[\Gr = \bigsqcup_{\lambda \in \bfXv^{+}} \Gr_{\lambda}^{+}.\]
    We will write $j_{\lambda}^{+} : \Gr_{\lambda}^{+} \hookrightarrow \Gr$ for the inclusion map. Note that $\Gr_{\lambda}^+$ is dense in $\Gr^{\lambda}$. As a result, we have that 
    \[\dim (\Gr_{\lambda}^+) = \langle \lambda, 2\rho \rangle\]
    and 
    \[\overline{\Gr_{\lambda}^+} = \bigcup_{\mu \preceq \lambda} \Gr_{\mu}^+. \]

    \subsection{Constructible Sheaves}

   Let $X$ be an ind-scheme of ind-finite type over $\C$.
   We will write $D_{\cons}^b (X, \k)$ for the bounded derived category of constructible sheaves on $X$ with respect to the analytic topology on $X$.
    Let $H$ be a group scheme of the form $H = \ilim H_k$ for some linear algebraic groups $H_k$.
    We denote the bounded derived category of $H$-equivariant constructible sheaves by $D_H^b (X, \k)$. 
    To ensure this is well-defined (cf., \cite{Ga01}), we assume the action is ``nice,'' i.e., every closed subscheme $Z$ of $X$ is contained in a larger closed subscheme $Z'$ such that $Z'$ is $H$-stable and the action of $H$ on $Z'$ factors through some $H_k$.
    We will write $P_H (X, \k)$ for the heart of the perverse $t$-structure on $D_H^b (X, \k)$. 
    For simplicity, all the derived categories of sheaves in this paper are triangulated; although, our arguments are equally valid when working with stable $\infty$-categories as well.

    We are primarily interested in the constructible derived category $D_c^b (\Gr, \k)$ and its equivariant variations. 
    Let $\lambda \in \bfXv^+$.
    We have three $\L^+ G$-equivariant perverse sheaves associated with $\lambda$:
    \[ \scrJ_! (\lambda, \k) \coloneq {}^p H^0 (j_{\lambda !} \uk_{\Gr^{\lambda}} [\langle \lambda, 2\rho \rangle]), \quad \scrJ_* (\lambda, \k) \coloneq {}^p H^0 (j_{\lambda *} \uk_{\Gr^{\lambda}} [\langle \lambda, 2\rho \rangle]), \quad \scrJ_{!*} (\lambda, \k) \coloneq j_{\lambda !*} \uk_{\Gr^{\lambda}} [\langle \lambda, 2\rho \rangle].\]
    Note that when $\k$ is a field, $ \scrJ_{!*} (\lambda, \k)$ is a simple perverse sheaf.
    We will recall some well-known facts about these sheaves.

    \begin{lemma}[\!\!{\cite[Proposition 8.1]{MV}}]\label{lem:eos_and_spherical}
        Let $\lambda \in \bfXv^+$.
        \begin{enumerate}
            \item  We have isomorphisms
        \[\scrJ_! (\lambda, \k) \cong \k \otimes_{\Z}^L \left( \scrJ_! (\lambda, \Z) \right) \qquad\text{and}\qquad \scrJ_* (\lambda, \k) \cong \k \otimes_{\Z}^L \left( \scrJ_* (\lambda, \Z) \right).\]
            \item There is an isomorphism 
         \[ \DD (\scrJ_! (\lambda, \k)) \cong \scrJ_* (\lambda, \k).\]
    \end{enumerate}
    \end{lemma}

    \begin{lemma}[\!\!{\cite[Remark 3.2]{BGMRR}}]\label{lem:hom_vanishing_for_spherical}
        Let $\lambda, \mu \in \bfXv^+$. We have
        \[\Ext_{P_{\L^+ G} (\Gr, \k)}^n (\scrJ_! (\lambda, \k), \scrJ_* (\mu, \k) ) \cong \begin{cases} \k & \text{if } n =0 \text{ and } \lambda = \mu; \\ 0 & \text{otherwise.}\end{cases}\]
    \end{lemma}

    \subsection{Convolution}

    Consider the diagram
    \[\begin{tikzcd}
    \Gr \times \Gr & \L G \times \Gr \arrow[l, "p"'] \arrow[r, "q"] & \L G \times^{\L^+ G} \Gr \arrow[r, "m"] & \Gr
    \end{tikzcd}\]
    where $p,q$ are quotient morphisms, and $m$ is induced by the $\L G$ action on $\Gr$.
    For $H \subseteq \L^+ G$, we will define a bifunctor
    \[(-) \star (-) : D_H^b (\Gr, \k) \times D_{\L^+ G}^b (\Gr, \k) \to D_H^b (\Gr, \k)\]
    as follows. Consider the action of $H \times \L^+ G$ on $\L G \times \Gr$ defined by
    \[(g_1,g_2) \cdot (h_1, h_2 \L^+ G) = (g_1 h_1 g_2^{-1}, g_2 h_2 \L^+ G). \]
    The functor $q^*$ induces an equivalence of categories
    \[D_{H}^b (\L G \times^{\L^+ G} \Gr, \k) \stackrel{\sim}{\to} D_{H \times \L^+ G}^b (\L G \times \L^+ G, \k).\]
    Hence, there exists a unique object $\scrF \tilde{\boxtimes} \scrG$ such that
    \[q^* (\scrF \tilde{\boxtimes} \scrG) = p^* (\scrF \boxtimes \scrG).\]
    Then the convolution product of $\scrF$ and $\scrG$ is defined by
    \[\scrF \star \scrG \coloneq m_* (\scrF \tilde{\boxtimes} \scrG).\]
    This construction makes $D_H^b (\Gr, \k)$ into a right $D_{\L^+ G}^b (\Gr, \k)$-module.
    In the special case where $H = \L^+ G$, convolution makes $D_{\L^+ G}^b (\Gr, \k)$ into a monoidal category.

    \begin{lemma}\label{lem:av_and_convolution}
        Let $H_1 \subseteq H_2 \subseteq \L^+ G$. The following diagrams commute up to natural isomorphism
        \[\begin{tikzcd}[column sep=small]
    {D_{H_1}^b (\Gr, \k) \times D_{\L^+ G}^b (\Gr, \k)} \arrow[d, "{\Av_{H_1, ?}^{H_2} \times \id}"'] \arrow[r, "\star"] & {D_{H_1}^b (\Gr, \k)} \arrow[d, "{\Av_{H_1, ?}^{H_2}}"] &  & {D_{H_2}^b (\Gr, \k) \times D_{\L^+ G}^b (\Gr, \k)} \arrow[r, "\star"] \arrow[d, "\For_{H_2}^{H_1} \times \id"'] & {D_{H_2}^b (\Gr, \k)} \arrow[d, "\For_{H_1}^{H_2}"] \\
    {D_{H_2}^b (\Gr, \k) \times D_{\L^+ G}^b (\Gr, \k)} \arrow[r, "\star"]                                               & {D_{H_2}^b (\Gr, \k)}                                   &  & {D_{H_1}^b (\Gr, \k) \times D_{\L^+ G}^b (\Gr, \k)} \arrow[r, "\star"]                                           & {D_{H_1}^b (\Gr, \k)}                              
    \end{tikzcd}\]
        where $? \in \{*,!\}$.
    \end{lemma}
    \begin{proof}
        The natural isomorphism can be deduced from a routine base change argument using the fact that $m$ is proper and $p,q$ are smooth.
    \end{proof}

    \begin{lemma}\label{lem:H_conv_t_exactness}
        Assume that $\k$ is a field. Let $H \subseteq \L^+ G$. Then
        \[(-) \star (-) : D_H^b (\Gr, \k) \times D_{\L^+ G}^b (\Gr, \k) \to D_H^b (\Gr, \k)\]
        is perverse $t$-exact.
    \end{lemma}
    \begin{proof}
        The perverse $t$-structure on $D_H^b (\Gr, \k)$ is defined by transporting the perverse $t$-structure on $D^b (\Gr, \k)$ along $\For^H$.
        In light of Lemma \ref{lem:av_and_convolution}, it then suffices to prove the lemma for $H = 1$. This case follows from \cite[Lemma 2.3]{BGMRR}.
    \end{proof}

    \begin{remark}
    When $\k$ is not a field, convolution on $D_{\L^+ G}^b (\Gr, \k)$ is not perverse $t$-exact.
    However, as discussed in \cite{MV}, one can still make $P_{\L^+ G} (\Gr, \k)$ into a (symmetric) monoidal category where the product is defined by $\scrF \:^p\!\star \scrG \coloneq {}^p H^0 (\scrF \star \scrG)$.
    \end{remark}

    \section{The Kirillov Model}
        The usual notion of Whittaker sheaves depends on the existence of a multiplicative local system $\Psi$ on $\A^1$ which satisfies $R\Gamma (\Psi) = R\Gamma_c (\Psi) = 0$.
    In the de Rham setting, one can take $\Psi$ to be the exponential $D$-module. Likewise, in the étale setting, one can take $\Psi$ to be an Artin--Schreier sheaf.
    Unfortunately, in the Betti setting, no such local system exists since $\A^1$ is contractible. 
    To remedy this problem, one can take inspiration from Laumon's homogenous Fourier transform \cite{Lau}.
    In this setting, one imposes an additional $\G_m$-equivariance on $\A^1$. There are essentially two different $\G_m$-equivariant sheaves on $\A^1$ that one can consider.
    Let $u : \G_m \to \A^1$ be the obvious inclusion. The sheaves $u_* \uk [1]$ and $u_! \uk [1]$ are known as the Fourier--Laumon kernels.
    Laumon \cite{Lau} proved that they satisfy a version of multiplicativity and that $R\Gamma (u_! \uk [1]) = R\Gamma_c (u_* \uk [1]) = 0$.
    Note that $R\Gamma (u_* \uk [1])$ and  $R\Gamma_c (u_! \uk [1])$ are both nonzero, so the vanishing is not symmetric.
 
    In this section, we will study a variation of Whittaker sheaves for general varieties that philosophically corresponds to looking at $(\A^1/\G_m, u_? \uk [1])$-equivariant sheaves where $? \in \{*,!\}$.
    This gives rise to the \emph{Kirillov model} introduced by Gaitsgory and Lysenko \cite{GaiWhit, GL}. The Kirillov model has the benefit of being defined in the Betti setting with $\Z$-coefficients; however, it does require our varieties to admit an additional $\G_m$-equivariance.
    Unlike the étale and de Rham settings, one has to consider both a $*$- and $!$-version of the Kirillov model corresponding to the lack of symmetry between $u_* \uk [1]$ and $u_! \uk [1]$.
    
    \subsection{Definition and First Properties}\label{subsec:kirillov_def} 
    Let $X$ be an ind-scheme of ind-finite type over $\C$. Let $H$ be a group scheme of the form $H = \ilim H_k$ for some linear algebraic groups $H_k$.
    Assume that $H$ acts nicely on $X$.
    Fix a non-degenerate additive character $\chi : H \twoheadrightarrow \G_a$, and denote its kernel by $K_{\chi}$.
    We will fix a $\G_m$-action on $H$, for example, one given by a cocharacter $\G_m \to H$. 
    We also assume that the $H$-action on $X$ extends to an $H \rtimes \G_m$-action on $X$. Define averaging functors,
    \[\Ind_{!}^{\chi} \coloneq \For_{K_{\chi} \rtimes \G_m}^{H \rtimes \G_m} \Av_{K_{\chi} \rtimes \G_m, !}^{H \rtimes \G_m} : D_{K_{\chi} \rtimes \G_m}^b (X, \k) \to D_{K_{\chi} \rtimes \G_m}^b (X, \k)\]
    \[\Ind_{*}^{\chi} \coloneq \For_{K_{\chi} \rtimes \G_m}^{H \rtimes \G_m} \Av_{K_{\chi} \rtimes \G_m, *}^{H \rtimes \G_m} : D_{K_{\chi} \rtimes \G_m}^b (X, \k) \to D_{K_{\chi} \rtimes \G_m}^b (X, \k).\]
    We define the $!$-\emph{Kirillov model} and $*$-\emph{Kirillov model} for $X$ by
    \[\Kir_{(H \rtimes \G_m, \chi), !} (X, \k) \coloneq \ker (\Ind_{!}^{\chi}) \qquad\text{and}\qquad \Kir_{(H \rtimes \G_m, \chi), *} (X, \k) \coloneq \ker (\Ind_{*}^{\chi}).\]
    The Kirillov model depends on the choice of $\chi$. We often omit the subscript $(H \rtimes \G_m, \chi)$ from the notation when obvious from context.
    We may also occasionally use $?$ in place of $!$ or $*$ in the subscript when discussing results that hold for both variations of the Kirillov model.
    Note that $\Kir_! (X, \k)$ and $\Kir_* (X, \k)$ are full triangulated subcategories of $D_{K_{\chi} \rtimes \G_m}^b (X, \k)$. 
    
    \begin{lemma}\label{lem:t_str_existence}
        The triangulated category $\Kir_? (X, \k)$ admits a $t$-structure obtained by restricting the perverse $t$-structure on $D_{K_{\chi} \rtimes \G_m}^b (X, \k)$.
    \end{lemma}
    \begin{proof}
        We will just verify the lemma for the $!$-Kirillov model. The $*$-variant can be obtained by a similar argument.
        We must check that for $\scrF \in \Kir_! (X, \k)$ that ${}^p \tau_{\leq 0} \scrF$ and ${}^p \tau_{\geq 0} \scrF$ are both in $\Kir_! (X, \k)$.
        First, observe that $\Ind_!^{\chi}$ is right $t$-exact, i.e., 
        \begin{equation}\label{eq:t_str_existence_1}
                \Ind_!^{\chi} ({}^p D_{K_{\chi} \rtimes \G_m}^b (X, \k)^{\leq 0}) \subseteq {}^p D_{K_{\chi} \rtimes \G_m}^b (X, \k)^{\leq 0} .
        \end{equation}
        Moreover, since $\G_a$ is unipotent of dimension 1, we also have that
        \begin{equation}\label{eq:t_str_existence_2}
             \Ind_!^{\chi} ({}^p D_{K_{\chi} \rtimes \G_m}^b (X, \k)^{\geq 0}) \subseteq {}^p D_{K_{\chi} \rtimes \G_m}^b (X, \k)^{\geq -1} .
        \end{equation}
        We can now consider the distinguished triangle induced from truncation
        \[\Ind_!^{\chi} ({}^p \tau_{\leq 0} \scrF) \to \Ind_!^{\chi} (\scrF) \to \Ind_!^{\chi} ({}^p \tau_{\geq 1}) \to .\]
        The middle term is $0$ by assumption, so there is an isomorphism 
        \[ \Ind_!^{\chi} ({}^p \tau_{\geq 1}) \cong \Ind_!^{\chi} ({}^p \tau_{\leq 0} \scrF) [1] \]
        The left-hand side of this isomorphism is in ${}^p D_{K_{\chi} \rtimes \G_m}^b (X, \k)^{\geq 0}$ by (\ref{eq:t_str_existence_2}).
        The right-hand side of the isomorphism is in ${}^p D_{K_{\chi} \rtimes \G_m}^b (X, \k)^{\leq -1}$ by (\ref{eq:t_str_existence_1}).
        As a result, we conclude that both sides are in fact 0. Hence, ${}^p \tau_{\leq 0} \scrF$ and ${}^p \tau_{\geq 0} \scrF$ are objects in $\Kir_! (X, \k)$.
    \end{proof}

    There are fully faithful $t$-exact functors
    \[\iota_! : \Kir_! (X, \k) \to D_{K_{\chi} \rtimes \G_m}^b (X, \k) \qquad\text{and}\qquad \iota_* : \Kir_* (X, \k) \to D_{K_{\chi} \rtimes \G_m}^b (X, \k)\]
    coming from $\Kir_! (X, \k)$ and $\Kir_* (X, \k)$ being full subcategories of $D_{K_{\chi} \rtimes \G_m}^b (X, \k)$.

    Since there is a natural isomorphism of functors $\DD \circ \Ind_!^{\chi} \cong \Ind_*^{\chi} \circ \DD$, one has that Verdier duality exchanges the $!$- and $*$-Kirillov models. 
    In particular, $\DD$ restricts to an equivalence of categories
    \[\DD : \Kir_! (X, \k)^{\op} \stackrel{\sim}{\to} \Kir_* (X, \k).\]

    \subsection{The Prototypical Example}

    Before we continue with the general theory, we will take a short interlude to focus on $\G_m$-equivariant sheaves on $\A^1$.
    We can define two maps
    \[\begin{tikzcd}
    \G_m \arrow[r, "u", hook] & \A^1 & \pt \arrow[l, "s"', hook']
    \end{tikzcd}\]
    as the obvious inclusion maps.
    The sheaves $u_* \uk [1]$ and $u_! \uk [1]$ are called the \emph{Fourier--Laumon kernels}. By \cite[Lemme 1.4]{Lau}, they satisfy
    \begin{equation}\label{eq:properties_of_FL_kernels}
        R\Gamma_c (u_* \uk [1]) \cong \Av^{\G_a}_! (u_* \uk [1]) = 0 \qquad\text{and}\qquad  R\Gamma (u_! \uk [1]) \cong \Av^{\G_a}_* (u_! \uk [1]) = 0. 
    \end{equation}

    Let $\sigma : \A^2 \to \A^1$ denote the map given by $(x,y) \mapsto x+y$. 
    We can make $\A^2$ into a $\G_m$-variety by $t \cdot (x,y) = (tx,ty)$. Under this action, $\sigma$ is $\G_m$-equivariant.
    There are two monoidal structures on $D_{\G_m}^b (\A^1, \k)$ given by $!$- and $*$-convolution,
    \[(-) \cstar (-) \coloneq \sigma_! \For^{\G_m^2}_{\G_m} (- \boxtimes -) : D_{\G_m}^b (\A^1, \k) \times D_{\G_m}^b (\A^1, \k) \to D_{\G_m}^b (\A^1, \k),\]
    \[(-) \sstar (-) \coloneq \sigma_* \For^{\G_m^2}_{\G_m} (- \boxtimes -) : D_{\G_m}^b (\A^1, \k) \times D_{\G_m}^b (\A^1, \k) \to D_{\G_m}^b (\A^1, \k).\]
    Since $\A^1$ is a commutative group scheme, both of these monoidal structures are symmetric.
    The monoidal unit for both $\cstar$ and $\sstar$ are given by $s_* \k$. 
    Verdier duality induces a symmetric monoidal equivalence of categories
    \[\DD : (D_{\G_m}^b (\A^1, \k)^{\op}, \cstar) \stackrel{\sim}{\to}  (D_{\G_m}^b (\A^1, \k), \sstar).\]

    \begin{lemma}\label{lem:two_out_of_three}
        There are isomorphisms
        \[u_! \uk [1] \cstar u_* \uk [1] \cong u_* \uk [1] \qquad\text{and}\qquad u_* \uk [1] \sstar u_! \uk [1] \cong u_! \uk [1].\]
    \end{lemma}
    \begin{proof}
        We will just prove the first isomorphism. The second follows from Verdier duality.
        There is an excision triangle
        \[s_* \k \to u_! \uk [1] \to \uk_{\A^1} [1] \to.\]
        We can apply $(-) \cstar u_* \uk [1]$ to produce a distinguished triangle
        \[u_* \uk [1] \to u_! \uk [1] \cstar u_* \uk [1] \to \uk_{\A^1} [1] \cstar u_* \uk [1] \to.  \]
        After forgetting the $\G_m$-equivariance, the last term can be identified with $\Av^{\G_a}_! (u_* \uk)$ which vanishes by (\ref{eq:properties_of_FL_kernels}).
        Therefore, $u_* \uk [1] \to u_! \uk [1] \cstar u_* \uk [1]$ is an isomorphism.
    \end{proof}

    \begin{example}\label{ex:kir_on_A1}
        We may consider the Kirillov model for $\A^1$. We claim that $\Kir_! (\A^1, \k)$ is generated as a triangulated category by $u_* \uk [1]$.
        Indeed, from (\ref{eq:properties_of_FL_kernels}), one has that $u_* \uk [1]$ is in $\Kir_! (\A^1, \k)$.
        On the other hand, suppose that $\scrF \in \Kir_! (\A^1, \k)$ and consider the excision triangle
        \[s_* s^! \scrF \to \scrF \to u_* u^* \scrF \to \] 
        It follows from (\ref{eq:properties_of_FL_kernels}) that $u_* u^* \scrF$ is also in $\Kir_! (\A^1, \k)$.
        In particular, we have isomorphisms
        \[0 = R\Gamma_c (\scrF) \cong R\Gamma_c (s_* s^! \scrF) \cong s^! \scrF.\]
        Therefore, $\scrF \cong u_* u^* \scrF$ which proves our claim.

        As a consequence, the functor
        \[ R\Hom (u_* \uk [1], -) : \Kir_! (\A^1, \k) \to D^b (\mod{\k})\]
        is a $t$-exact equivalence of categories.

        A symmetric argument shows that $\Kir_* (\A^1, \k)$ is generated $u_! \uk [1]$.
    \end{example}

    \subsection{Averaging and Translation Functors}

    We return to the general setting from \S\ref{subsec:kirillov_def}.
    We can make $D_{K_{\chi} \rtimes \G_m}^b (X, \k)$ into a module over $(D_{\G_m}^b (\A^1, \k), \cstar)$ and $(D_{\G_m}^b (\A^1, \k), \sstar)$  using both $*$- and $!$-convolution.
    Let $a : H \times X \to X$ denote the action map.
    Note that the pullback along $\chi$ defines an equivalence of categories $\chi^* : D_{\G_m}^b (\A^1, \k) \stackrel{\sim}{\to} D_{K_{\chi} \rtimes \G_m} (H, \k)$.
    Let $\tilde{a} : H \times^{K_{\chi} \rtimes \G_m} X \to X$ be the map induced from the action map $a$.
    We can then define
    \[(-) \cstar (-) \coloneq \tilde{a}_! (\chi^* (-) \tilde{\boxtimes} -) : D_{\G_m}^b (\A^1, \k) \times D_{K_{\chi} \rtimes \G_m}^b (X, \k) \to D_{K_{\chi} \rtimes \G_m}^b (X, \k),\]
    \[(-) \sstar (-) \coloneq \tilde{a}_* (\chi^* (-) \tilde{\boxtimes} -) : D_{\G_m}^b (\A^1, \k) \times D_{K_{\chi} \rtimes \G_m}^b (X, \k) \to D_{K_{\chi} \rtimes \G_m}^b (X, \k).\]
    Here $\chi^* (\scrF) \tilde{\boxtimes} (-)$ is the functor sending an object $\scrG$ to the unique object in $D_{K_{\chi} \rtimes \G_m}^b (H \times^{K_{\chi} \rtimes \G_m} X, \k)$ whose pullback along $H \times X \to H \times^{K_{\chi} \rtimes \G_m} X$ is isomorphic to $\chi^* \scrF \boxtimes \scrG \in D_{(K_{\chi} \rtimes \G_m)^2}^b (X, \k)$.

    \begin{lemma}\label{lem:existence_of_kirillov_averaging}\qquad
        \begin{enumerate}
        \item The tautological functor $\iota_! : \Kir_! (X, \k) \hookrightarrow D_{K_{\chi} \rtimes \G_m}^b (X, \k)$ admits a right adjoint
        \[\Av^!_{\Kir} : D_{K_{\chi} \rtimes \G_m}^b (X, \k) \to \Kir_! (X, \k), \qquad \scrF \mapsto u_* \uk [1] \cstar \scrF.\]
        Moreover, there is a functorial distinguished triangle
        \[\iota_! \Av^!_{\Kir} (\scrF) \to \scrF \to \Ind_!^{\chi} (\scrF) \to\]
        for $\scrF \in D_{K_{\chi} \rtimes \G_m}^b (X, \k)$.
        \item The tautological functor $\iota_* : \Kir_* (X, \k) \hookrightarrow D_{K_{\chi} \rtimes \G_m}^b (X, \k)$ admits a left adjoint
        \[\Av^*_{\Kir} : D_{K_{\chi} \rtimes \G_m}^b (X, \k) \to \Kir_* (X, \k), \qquad \scrF \mapsto u_! \uk [1] \sstar \scrF.\]
        Moreover, there is a functorial distinguished triangle
        \[\Ind_*^{\chi} (\scrF) \to \scrF \to \iota_* \Av_*^{\Kir} (\scrF) \to\]
        for $\scrF \in D_{K_{\chi} \rtimes \G_m}^b (X, \k)$.
        \end{enumerate}
    \end{lemma}
    \begin{proof}
        We will just prove the statement for the right adjoint of $\iota_!$. The left adjoint of $\iota_*$ is similar.
        Let $\scrF \in D_{K_{\chi} \rtimes \G_m}^b (X, \k)$. Consider the excision triangle
        \[u_* \uk [1] \to s_* \k \to \uk_{\A^1} [2] \to.\]
        We can apply $(-) \cstar \scrF$ to construct a distinguished triangle
        \begin{equation}\label{eq:existence_of_kirillov_av_1}
            u_* \uk [1] \cstar \scrF \to \scrF \to \uk_{\A^1} [2] \cstar \scrF \to.
        \end{equation}
        The last term of (\ref{eq:existence_of_kirillov_av_1}) can be identified with $\Ind_{!}^{\chi} (\scrF)$.
        We can then apply $\Ind_!^{\chi} (\scrF)$ to (\ref{eq:existence_of_kirillov_av_1}) to get a new distinguished triangle
        \[\Ind_!^{\chi} (u_* \uk [1] \cstar \scrF) \to \Ind_!^{\chi} (\scrF) \stackrel{\sim}{\to} \Ind_!^{\chi} \Ind_!^{\chi} (\scrF) \to.\]
        The second morphism is an isomorphism since $H/K_{\chi}$ is unipotent. As a result, we must have that $\Ind_!^{\chi} (u_* \uk [1] \cstar \scrF) = 0$, and hence $u_* \uk [1] \cstar \scrF \in \Kir_! (X, \k)$.
        It follows from (\ref{eq:existence_of_kirillov_av_1}) that if $\scrF \in \Kir_! (X, \k)$, then $u_* \uk [1] \cstar \scrF \cong \scrF$. 
        One can readily check from these observations that $u_* \uk [1] \star (-) : D_{K_{\chi} \rtimes \G_m}^b (X, \k) \to \Kir_! (X, \k)$ is indeed the right adjoint of $\iota_!$.
    \end{proof}

    Define a functor $T_{!}^* : \Kir_! (X, \k) \to \Kir_* (X, \k)$ by the composition
    \[\begin{tikzcd}
{\Kir_! (X, \k)} \arrow[r, "\iota_!"] & {D_{K_{\chi} \rtimes \G_m}^b (X, \k)} \arrow[r, "\Av_{\Kir}^*"] & {\Kir_* (X, \k).}
\end{tikzcd}\]
    This functor admits a right adjoint $T_{*}^! : \Kir_* (X, \k) \to \Kir_! (X, \k)$ given by the composition
        \[\begin{tikzcd}
    {\Kir_* (X, \k)} \arrow[r, "\iota_*"] & {D_{K_{\chi} \rtimes \G_m}^b (X, \k)} \arrow[r, "\Av_{\Kir}^!"] & {\Kir_! (X, \k).}
    \end{tikzcd}\]
    In light of Lemma \ref{lem:existence_of_kirillov_averaging}, $T_!^*$ (resp. $T_*^!$) is given by the convolution $u_! \uk [1] \sstar (-)$ (resp. $u_* \uk [1] \cstar (-)$).
    We call $T_!^*$ and $T_*^!$ the \emph{translation functors}. Since Verdier duality exchanges $*$-convolution with $!$-convolution, it follows that there is a natural isomorphism $T_*^! \circ \DD \cong \DD \circ T_!^*$.

    \begin{proposition}\label{prop:translation_functors_are_equivs}
        The functor $T_{!}^* : \Kir_! (X, \k) \to \Kir_* (X, \k)$ is a perverse $t$-exact equivalence of categories with inverse $T_*^!$.
    \end{proposition}
    \begin{proof} 
        Let $\scrF \in \Kir_! (X, \k)$. By Lemma \ref{lem:existence_of_kirillov_averaging}, there is a distinguished triangle
        \[\Ind_*^{\chi} \scrF \to \scrF \to u_! \uk [1] \sstar \scrF \to.\]
        We can then apply $u_* \uk [1] \cstar (-)$ to the above distinguished triangle to yield 
        \[u_* \uk [1] \cstar \Ind_*^{\chi} \scrF \to u_* \uk [1] \cstar \scrF \to u_* \uk [1] \cstar (u_! \uk [1] \sstar \scrF) \to. \]
        Note that since $\scrF \in \Kir_! (X, \k)$, we have that $u_* \uk [1] \cstar \scrF \cong \scrF$.
        We claim that $\scrF \to u_* \uk [1] \cstar (u_! \uk [1] \sstar \scrF)$ is an isomorphism. It suffices to show that $u_* \uk [1] \cstar \Ind_*^{\chi} \scrF  = 0$.
        Again by Lemma \ref{lem:existence_of_kirillov_averaging}, there is a distinguished triangle
        \[u_* \uk [1] \cstar \Ind_*^{\chi} \scrF \to \Ind_*^{\chi} \scrF \stackrel{\eta}{\to} \Ind_!^{\chi} \Ind_*^{\chi} \scrF \to.\]
        On the other hand, we have a morphism $\epsilon : \Ind_!^{\chi} \Ind_*^{\chi} \scrF \to \Ind_*^{\chi}$ induced by the counit of the adjunction $(\Av^{H \rtimes \G_m, !}_{K_{\chi} \rtimes \G_m}, \For^{H \rtimes \G_m}_{K_{\chi} \rtimes \G_m})$.
        By the zigzag identity, $\eta \circ \epsilon \cong \id$. Moreover, $\epsilon$ is necessarily an isomorphism since $\G_a$ is contractible, and so $\eta$ is an isomorphism.
        We deduce that $u_* \uk [1] \cstar \Ind_*^{\chi} \scrF = 0$, and hence that $\scrF \to u_* \uk [1] \cstar (u_! \uk [1] \sstar \scrF)$ is an isomorphism.
        By a final application of Lemma \ref{lem:existence_of_kirillov_averaging}, we see that $u_* \uk [1] \cstar (u_! \uk [1] \sstar \scrF)$ can be naturally identified with $(T_*^! \circ T_!^*) (\scrF)$.
        Therefore, $T_*^!$ is a left inverse for $T_!^*$. The argument that $T_*^!$ is also a right inverse follows from a symmetric argument and is omitted.

        It remains to show that $T_!^*$ is perverse $t$-exact. Note that $\iota_!$ is $t$-exact and $\Av_{\Kir}^*$ is right $t$-exact because $\iota_*$ is $t$-exact.
        As a result, $T_!^*$ is right $t$-exact. Since $T_!^*$ is an equivalence of categories, this is enough to deduce that $T_!^*$ is $t$-exact.
    \end{proof}

    \begin{example}\label{ex:translation_functor_for_A1}
        We again consider the example of the Kirillov model for $\A^1$.
        By Lemma \ref{lem:two_out_of_three}, there is an isomorphism 
        \[ T_!^* (u_* \uk [1]) \cong u_! \uk[1] \sstar u_* \uk [1] \cong u_! \uk[1]. \]
    \end{example}

    \subsection{Sheaf Functors}

    Let $f : X \to Y$ be a bounded $H \rtimes \G_m$-equivariant morphism of ind-schemes.
    It follows from functorality and base change that $f_!$ and $f^*$ commute with $\Ind_!^{\chi}$.
    As a result, these functors restrict to functors
    \[f_! : \Kir_! (X, \k) \to \Kir_! (Y, \k) \qquad\text{and}\qquad f^* : \Kir_! (Y,\k) \to \Kir_! (X, \k). \]
    Likewise, $f_*$ and $f^!$ commute with $\Ind_*^{\chi}$, so there are restricted functors
    \[f_* : \Kir_* (X, \k) \to \Kir_* (Y, \k) \qquad\text{and}\qquad f^! : \Kir_* (Y,\k) \to \Kir_* (X, \k). \]
    Unlike the Whittaker model, the functors $f^!$ and $f_*$ do not preserve $\Kir_!$ and likewise $f^*$ and $f_!$ do not preserve $\Kir_*$.
    Nonetheless, these functors do admit well-defined adjoints.
    
    We will start with the $!$-Kirillov model. The functors $f_!$ and $f^*$ admit right adjoints denoted by
    \[f_{\Kir}^! : \Kir_! (Y, \k) \to \Kir_! (X, \k) \qquad\text{and}\qquad f_*^{\Kir} : \Kir_! (X,\k) \to \Kir_! (Y, \k).\]
    Explicitly, they are given by the compositions
    \[\begin{tikzcd}
{f_{\Kir}^! :  \Kir_! (Y, \k)} \arrow[r, "\iota_!", hook] & {D_{K_{\chi} \rtimes \G_m}^b (Y, \k)} \arrow[r, "f^!"] & {D_{K_{\chi} \rtimes \G_m}^b (X, \k)} \arrow[r, "\Av_{\Kir}^!"] & {\Kir_! (X, \k),}
\end{tikzcd} \]
    \[\begin{tikzcd}
{f_*^{\Kir} : \Kir_! (X, \k)} \arrow[r, "\iota_!", hook] & {D_{K_{\chi} \rtimes \G_m}^b (X, \k)} \arrow[r, "f_*"] & {D_{K_{\chi} \rtimes \G_m}^b (Y, \k)} \arrow[r, "\Av_{\Kir}^!"] & {\Kir_! (Y, \k).}
\end{tikzcd}\]
    Since $\iota_!$ is fully faithful, we have that $\Av_{\Kir}^! \circ \iota_! \cong \id$. 
    It is then easy to check from this observation that $f_! \dashv f_{\Kir}^!$ and $f^* \dashv f_*^{\Kir}$. 

    Similarly, in the $*$-Kirillov model, the functors $f_*$ and $f^!$ admit left adjoints. These will be denoted by
    \[f_{\Kir}^* : \Kir_* (Y, \k) \to \Kir_* (X, \k) \qquad\text{and}\qquad f_!^{\Kir} : \Kir_* (X,\k) \to \Kir_* (Y, \k).\]
    They are defined by the compositions
        \[\begin{tikzcd}
{f_{\Kir}^* : \Kir_* (Y, \k)} \arrow[r, "\iota_*", hook] & {D_{K_{\chi} \rtimes \G_m}^b (Y, \k)} \arrow[r, "f^*"] & {D_{K_{\chi} \rtimes \G_m}^b (X, \k)} \arrow[r, "\Av_{\Kir}^*"] & {\Kir_* (X, \k),}
\end{tikzcd} \]
    \[\begin{tikzcd}
{f_!^{\Kir} : \Kir_* (X, \k)} \arrow[r, "\iota_*", hook] & {D_{K_{\chi} \rtimes \G_m}^b (X, \k)} \arrow[r, "f_!"] & {D_{K_{\chi} \rtimes \G_m}^b (Y, \k)} \arrow[r, "\Av_{\Kir}^*"] & {\Kir_* (Y, \k).}
\end{tikzcd}\]

    \begin{lemma}\label{lem:kir_functors_commute_with_DD}
        With the same notation as above, there are natural isomorphisms
        \[\DD \circ f_! \cong f_* \circ \DD, \qquad \DD \circ f^* \cong f^! \circ \DD, \qquad \DD \circ f_*^{\Kir} \cong f_!^{\Kir} \circ \DD, \qquad \DD \circ f_{\Kir}^! \cong f_{\Kir}^* \circ \DD.\]
    \end{lemma}
    \begin{proof}
        The first two isomorphisms are immediate. 
        The latter two follows from the fact that $\DD \circ \Av_{\Kir}^! \cong \Av_{\Kir}^* \circ \DD$ which can be deduced from the definitions. 
    \end{proof}

    \begin{definition}\label{def:sh_functors_commute_with_translation}
        We say that $f : X \to Y$ admits an \emph{$H \rtimes \G_m$-equivariant compactification} if there exists an ind-scheme of ind-finite type $\overline{X}$ with a nice $H \rtimes \G_m$-action such that $f$ factors as
        $f = \pi \circ j$ where $j : X \to \overline{X}$ is an $H \rtimes \G_m$-equivariant open embedding and $\pi : \overline{X} \to Y$ is an $H \times \G_m$-equivariant proper map.
    \end{definition}

    \begin{lemma}\label{lem:translation_and_sh_functors}
        Assume that $f : X \to Y$ admits an $H \rtimes \G_m$-equivariant compactification. Then there are natural isomorphisms
         \[T_!^* \circ f_! \cong f_!^{\Kir} \circ T_!^*, \qquad T_!^* \circ f^* \cong f_{\Kir}^* \circ T_!^*, \qquad T_!^* \circ f_*^{\Kir} \cong f_* \circ T_!^*, \qquad T_!^* \circ f_{\Kir}^! \cong f^! \circ T_!^*.\]
    \end{lemma}
    \begin{proof}
        Since $f$ admits a compactification, we only have to check these isomorphisms hold when $f$ is proper or $f$ is an open embedding. 

        First, assume that $f$ is proper. It is easy to check from functorality that $f_* = f_! : D_{K_{\chi} \rtimes \G_m}^b (X, \k) \to D_{K_{\chi} \rtimes \G_m}^b (Y, \k)$ is $D_{\G_m}^b (\A^1, \k)$-linear with respect to the $\sstar$-action.
        In this case, $f_!^{\Kir} = f_!$ and $f_*^{\Kir} = f_*$, and so we can conclude that $T_!^* \circ f_! \cong f_!^{\Kir} \circ T_!^*$ and $T_!^* \circ f_*^{\Kir} \cong f_* \circ T_!^*$.
        By taking adjoints and using Proposition \ref{prop:translation_functors_are_equivs}, we deduce that $T_!^* \circ f^* \cong f_{\Kir}^* \circ T_!^*$ and $T_!^* \circ f_{\Kir}^! \cong f^! \circ T_!^*$ as well.

        Now assume that $f$ is open. By base change, $f^* = f^! : D_{K_{\chi} \rtimes \G_m}^b (Y, \k) \to D_{K_{\chi} \rtimes \G_m}^b (X, \k)$ is $D_{\G_m}^b (\A^1, \k)$-linear with respect to the $\sstar$-action.
        Moreover, $f_{\Kir}^* = f^*$ and $f_{\Kir}^! = f^!$. It then follows from the definitions that $T_!^* \circ f^* \cong f_{\Kir}^* \circ T_!^*$ and $T_!^* \circ f_{\Kir}^! \cong f^! \circ T_!^*$.
        The remaining isomorphisms follow from taking adjoints and Proposition \ref{prop:translation_functors_are_equivs}.
    \end{proof}

    \begin{lemma}\label{lem:recollement}
        Let $U \subseteq X$ be an $H \times \G_m$-stable open subset with complement $Z \subseteq X$. Write $j : U \hookrightarrow X$ and $i : Z \hookrightarrow X$ denote the inclusion maps.
        There are recollement diagrams
        \[\begin{tikzcd}
{\Kir_! (Z, \k)} \arrow[r, "i_!"] & {\Kir_! (X, \k)} \arrow[r, "j^*"] \arrow[l, "i^*"', bend right] \arrow[l, "i_{\Kir}^!", bend left] & {\Kir_! (U, \k)} \arrow[l, "j_!"', bend right] \arrow[l, "j_*^{\Kir}", bend left]
\end{tikzcd}\]
\[\begin{tikzcd}
{\Kir_* (Z, \k)} \arrow[r, "i_*"] & {\Kir_* (X, \k)} \arrow[r, "j^!"] \arrow[l, "i_{\Kir}^*"', bend right] \arrow[l, "i^!", bend left] & {\Kir_* (U, \k)} \arrow[l, "j_!^{\Kir}"', bend right] \arrow[l, "j_*", bend left]
\end{tikzcd}\]
    \end{lemma}
    \begin{proof}
        We will just prove that diagram involving the $!$-Kirillov models forms a recollement diagram. The argument for the $*$-Kirillov variant is similar.
        There are only two non-trivial axioms that must be checked.
        \begin{enumerate}
            \item The functor $j_*^{\Kir} : \Kir_! (U, \k) \to \Kir_! (X, \k)$ is fully faithful.
            \item For $\scrF \in \Kir_! (X, \k)$, there is a distinguished triangle
            \[i_! i_{\Kir}^! \scrF \to \scrF \to j_*^{\Kir} j^* \scrF \to.\]
        \end{enumerate}
        
        For (1), by Lemma \ref{lem:translation_and_sh_functors}, there is a natural isomorphism $j_*^{\Kir} \cong T_*^! \circ j_* \circ T_!^*$. The full faithfulness of $j_*^{\Kir}$ then follows from $j_*$ being fully faithful and Proposition \ref{prop:translation_functors_are_equivs}.

        We can now prove (2). There is a distinguished triangle in $\Kir_* (X)$ induced by excision
        \[i_! i^! T_!^* \scrF \to T_!^* \scrF \to j_* j^* T_!^* \scrF \to\]
        We can apply $T_*^!$ to the above triangle to yield
        \[i_! T_*^! i^! T_!^* \scrF \to T_*^! T_!^* \scrF \to T_*^! j_* T_!^* j^* \scrF \to .\]
        By Lemma \ref{lem:translation_and_sh_functors} and Proposition \ref{prop:translation_functors_are_equivs}, this triangle can be rewritten as
         \[i_! i_{\Kir}^! \scrF \to \scrF \to j_*^{\Kir} j^* \scrF \to.\]
    \end{proof}

    \begin{lemma}\label{lem:costandards_and_perversity}
        Let $f : X \hookrightarrow Y$ be a locally closed inclusion of an affine $H \rtimes \G_m$-subvariety. 
        Then $f_*^{\IW}$ is perverse $t$-exact.
    \end{lemma}
    \begin{proof}
        Since $f$ is affine, the morphism $f_* : \Kir_* (X, \k) \to \Kir_* (Y, \k)$ is automatically perverse $t$-exact (see \cite[Corollaire 4.1.3]{BBD}). The result then follows from Proposition \ref{prop:translation_functors_are_equivs} and Lemma \ref{lem:translation_and_sh_functors}.
    \end{proof}

    Let $\k \to \k'$ be a ring homomorphism. 
    Since extension of scalars $\k' \otimes_{\k}^L (-)$ commutes with $\Ind_!^{\chi}$ and $\Ind_*^{\chi}$, we can consider the restriction of extension of scalars to the $*$- and $!$-Kirillov models.
    \[\k' \otimes_{\k}^L (-) : \Kir_! (X, \k) \to \Kir_! (X, \k') \qquad\text{and}\qquad \k' \otimes_{\k} (-) : \Kir_* (X, \k) \to \Kir_* (X, \k'). \]
    Moreover, $\k' \otimes_{\k}^L (-)$ commutes with the pullback and pushforward functors considered in this section. 

    \begin{lemma}\label{lem:eos_commutes_with_sh_functors}\quad
        \begin{enumerate}
            \item $\k' \otimes_{\k}^L (-)$ commutes with $\iota_!$, $\Av_{\Kir}^!$, and $T_!^*$.
            \item If $f : X \to Y$ admits an $\H \rtimes \G_m$-equivariant compactification, then $\k' \otimes_{\k} (-)$ commutes with $f_!$, $f^*$, $f_*^{\Kir}$, and $f_{\Kir}^!$.
        \end{enumerate}
        Obvious variations of (1) and (2) hold for the $*$-Kirillov model as well.
    \end{lemma}
    \begin{proof}
        Statement (1) follows from the definitions.
        Finally, (2) follows from (1) using the description of $f_*^{\Kir}$ and $f_{\Kir}^!$ from Lemma \ref{lem:translation_and_sh_functors}.
    \end{proof}

    \subsection{Riemann--Hilbert Correspondence}\label{subsec:rh}
    For an ind-variety $X$ with a nice $H$-action, we denote the bounded derived category of holonomic $D$-modules by $\Dmod_H^{\hol} (X)$. We will also write $\Dmod_H^{\rh} (X)$ for the full subcategory of $\Dmod_H^{\hol} (X)$
    consisting of regular holonomic $D$-modules.

Let $\k = \C$. We can apply the Riemann--Hilbert correspondence to change our sheaf-theoretic setting to that of regular holonomic $D$-modules.
    In particular, there is a commutative square
        \[\begin{tikzcd}
{D_{K_{\chi} \rtimes \G_m}^b (X, \C)} \arrow[d, "\sim"] \arrow[r, "\Ind_!^{\chi}"] & {D_{H \rtimes \G_m}^b (X, \C)} \arrow[d, "\sim"] \\
\Dmod_{K_{\chi} \rtimes \G_m}^{\rh} (X) \arrow[r, "\Ind_!^{\chi}"]                  & \Dmod_{H \rtimes \G_m}^{\rh} (X).                 
\end{tikzcd}\]
    The definition of the Kirillov model can be made regardless of sheaf-theoretic setting, so we can consider the $D$-module variant,
    \[\Kir_*^{\dR} (X) \coloneq \ker \left(  \Ind_*^{\chi} : \Dmod_{K_{\chi} \rtimes \G_m}^{\hol} (X) \to \Dmod_{K_{\chi} \rtimes \G_m}^{\hol} (X)\right).\]
    Here we are using holonomic $D$-modules so that we can later compare with the Whittaker model.  
    Then the above commutative square gives a $t$-exact fully faithful functor
    \begin{equation}\label{eq:kir_and_deRham}
        \Kir_* (X, \C) \hookrightarrow \Kir_*^{\dR} (X).
    \end{equation}

    On the other hand, in the de Rham setting there is an exponential $D$-module, denoted $\exp \in \Dmod^{\hol} (\A^1)$. 
    We define the \emph{Whittaker model} for $X$,
    \[\Whit^{\dR} (X) \coloneq \Dmod_{(H, \chi^* \exp) }^{\hol} (X),\]
    as the category of $(H, \chi^* \exp)$-equivariant holonomic $D$-modules on $X$. 

    Consider the composition of functors
    \[\begin{tikzcd}
\Whit^{\dR} (X) \arrow[r, hook] & \Dmod_{K_{\chi} }^{\hol} (X) \arrow[r, "\Av_*^{\G_m}"] & \Dmod_{K_{\chi} \rtimes \G_m }^{\hol} (X)
\end{tikzcd}\]
    where the first arrow is the forgetful functor. Since $\Av_*^{\G_m}$ commutes with $\Av_*^{\G_a}$, it is easy to see that the image of this composition belongs to $\Kir_*^{\dR} (X)$.
    In particular, there is a functor
    \begin{equation}\label{eq:whit_vs_kir}
        \Whit^{\dR} (X) \to \Kir_*^{\dR} (X).
    \end{equation}
    It is shown in \cite[\S 1.6.3]{GaiWhit} (see also \cite[Proposition A.3.2]{GL}) that (\ref{eq:whit_vs_kir}) is an equivalence of categories.

    Composing the functors (\ref{eq:kir_and_deRham}) and (\ref{eq:whit_vs_kir}), we obtain a fully faithful functor
    \begin{equation*}
        \mathcal{S} : \Kir_* (X, \C) \hookrightarrow \Whit^{\dR} (X).
    \end{equation*}
    Note that $\mathcal{S}$ commutes with $!$-pullbacks and $*$-pushforwards.

    In general, there is no reason to believe that $\mathcal{S}$ is essentially surjective. We will later see that this holds at least for the affine Grassmannian.
    The author expects that $\mathcal{S}$ is $t$-exact in general, but such a proof will not be provided. In the special case of the affine Grassmannian, the $t$-exactness will be proved in the next section.

    \section{Iwahori--Whittaker Sheaves}
    
    The goal of this section is to apply the Kirillov model in two important cases: partial (finite) flag varieties and the affine Grassmannian.

    \subsection{Whittaker Sheaves on Finite Flag Varieties}

    Let $P \subset G$ be a parabolic subgroup containing $B^{-}$ with Levi subgroup $L$.
    Let $W_L \subset W$ be the Weyl group of $L$ with respect to $T \subseteq L$, and denote the set of minimal length coset representatives for $W/W_L$ by $W^L$.

    Consider the character $\chi : U^+ \to \G_a$ defined by the composition
    \[U^+ \to U^+ / [U^+, U^+] \stackrel{\prod_{\alpha} u_{\alpha}}{\longrightarrow} \prod_{\alpha \in \Phi_s} \G_a \stackrel{\textnormal{sum}}{\to} \G_a.\]
    Write $U_{\chi}^+ = \ker \chi$. The cocharacter $\zeta : \G_m \to T$ affords an action of $U^+ \rtimes \G_m$ on the partial flag variety $G/P$. 

    Consider the $!$-Kirillov model on $G/P$ with respect to $(U^+ \rtimes \G_m, \chi)$, denoted
    \[D_{\Wh, !}^b (G/P, \k) \coloneq \Kir_{(U^+ \rtimes \G_m, \chi), !} (G/P, \k).\]

    The Bruhat decomposition gives a stratification of $G/P$ into $B^{+}$-orbits,
    \[G/P = \bigsqcup_{w \in W^L} X_w^P\]
    where $X_w^P = B^{+} \dot{w} P/P$ for some lift $\dot{w} \in N_G (T)$ of $w$. 
    Note that the $X_w^P$'s are locally closed affine spaces.

    Before proceeding with the general theory, we can focus on the case of $P = B^{-}$. Let $h_e : X_e^{B^{-}} \hookrightarrow G/B^{-}$ denote the inclusion map.
    Note that $X_e^{B^{-}} \cong U^{+}$ as $U^{+} \rtimes \G_m$-varieties. 
    As a result, $\chi$ induces an isomorphism 
    \[ \overline{\chi} : U_{\chi}^+ \rtimes \G_m \backslash X_e^{B^{-}} \stackrel{\sim}{\to} \G_m \backslash \A^1\]
    where $\G_m$ acts on $\A^1$ by dilation. We then have a functor
    \[ \Kir_! (\A^1, \k) \stackrel{\overline{\chi}^*}{\longrightarrow} \Kir_! (X_e^{B^{-}}, \k) \stackrel{h_{e!}}{\longrightarrow} D_{\Wh, !}^b (G/B^{-}, \k). \]

    \begin{proposition}\label{prop:strata_desc_of_Wh_for_flag_varieties}\quad
        \begin{enumerate}
            \item If $P \supsetneq B^{-}$, then $\Kir_! (X_w^P, \k) = 0$. In particular, $D_{\Wh, !}^b (G/P, \k) = 0$.
            \item If $P = B^{-}$, then the following composition
            \[\begin{tikzcd}
D^b (\mod{\k}^{\fg}) \arrow[r, "\sim"',"\ref{ex:kir_on_A1}"] & {\Kir_! (\A^1, \k)} \arrow[r, "\overline{\chi}^*"] & {\Kir_! (X_e^{B^{-}}, \k)} \arrow[r, "h_{e!}"] & {D_{\Wh, !}^b (G/B^{-}, \k)}
\end{tikzcd}\]
            is a $t$-exact equivalence of categories.
        \end{enumerate}
    \end{proposition}
    \begin{proof}
        Let $\scrF \in \Kir_! (X_w^P, \k)$. Note that for all $\beta \in \Phi^+$, we have that $U_{\beta} \dot{w} P/P$ is stable under the $\G_m$-action.
        If $P \neq B^{-}$ or $w \neq e$, we can find a simple root $\alpha \in \phi^+$ such that $U_{\alpha} \dot{w} P/P \subseteq \dot{w}P/P$.
        Since $U^+ = U_{\chi}^+ U_{\alpha}$, this subset inclusion forces $\scrF$ to be $U^+$-equivariant.
        Note that $\For_{U_{\chi}^+ \rtimes \G_m}^{U^+ \rtimes \G_m}$ is fully faithful since $U^+/U_{\chi}^+ \cong \G_a$ is unipotent. As a result, the natural morphism $\Av_{U_{\chi}^+ \rtimes \G_m, !}^{U^+ \rtimes \G_m} \For_{U_{\chi}^+ \rtimes \G_m}^{U^+ \rtimes \G_m} \stackrel{\sim}{\to} \id$ is an isomorphism.
        Since $\scrF$ occurs in the essential image of $\For_{U_{\chi}^+ \rtimes \G_m}^{U^+ \rtimes \G_m}$, we must have that $\Av_{U_{\chi}^+, ! \rtimes \G_m}^{U^+ \rtimes \G_m} \scrF = 0$ if and only if $\scrF = 0$.

        We have now shown (1) and the equivalence $\Kir_! (X_e^{B^-}, \k) \stackrel{\sim}{\to} D_{\Wh, !}^b (G/B^{-}, \k)$. It remains to show that $D^b (\mod{\k}^{\fg}) \to \Kir_! (X_e^{B^{-}}, \k)$ is a $t$-exact equivalence of categories.
        This follows from the discussion found in \cite[\S 5.8.4]{Sa1}. 
      \end{proof}

    We could have alternatively considered the $*$-Kirillov model on $G/P$,
    \[D_{\Wh, *}^b (G/P, \k) \coloneq \Kir_{(U^+ \rtimes \G_m, \chi), *} (G/P, \k).\]
    The constructions and results of this section have obvious variants for the $*$-Kirillov model as well.

    \subsection{Iwahori--Whittaker Sheaves on Affine Grassmannians}

    Let $I^+ \subset \L^+ G$ denote the Iwahori subgroup associated with $B^+$, and let $I_u^+$ denote its pro-unipotent radical.
    Denote $\chi_{I^+}$ the composition of $\chi : U^{+} \to \G_a$ with the projection $I_u^+ \twoheadrightarrow U^+$.
    Let $K_{\chi} = \ker (\chi_{I^+}) \subset I_u^+$ which is a normal subgroup of $I_u^+$ such that $I_u^+/K_{\chi} \cong \G_a$. 
    The cocharacter $\zeta : \G_m \to T$ affords an action of $I_u^+ \rtimes \G_m$ on $\Gr$.
    
    Define the categories of $!$- and $*$-\emph{Iwahori--Whittaker sheaves} on $\Gr$ via the Kirillov model,
    \[D_{\IW, !}^b (\Gr, \k) \coloneq \Kir_{(I_u^+ \rtimes \G_m, \chi_{I^+}), !} (\Gr,\k) \qquad\text{and}\qquad D_{\IW, *}^b (\Gr, \k) \coloneq \Kir_{(I_u^+ \rtimes \G_m, \chi_{I^+}), *} (\Gr,\k).\]  
    The perverse $t$-structure on $D_{K_{\chi} \rtimes \G_m}^b (\Gr, \k)$ restricts to $t$-structures on $D_{\IW, *}^b (\Gr, \k)$ and $D_{\IW, !}^b (\Gr, \k)$. We denote the heart of these $t$-structure by
    \[P_{\IW, !} (\Gr, \k) \qquad\text{and}\qquad P_{\IW, *} (\Gr, \k).\]

    We can apply Lemma \ref{lem:av_and_convolution} to the situation where $H_1 = K_{\chi} \rtimes \G_m$ and $H_2 = I_u^+ \rtimes \G_m$ to show that convolution induces bifunctors
    \[D_{\IW, !}^b (\Gr, \k) \times D_{\L^+ G}^b (\Gr, \k) \stackrel{\star}{\to} D_{\IW, !}^b (\Gr, \k) \quad\text{and}\quad D_{\IW, *}^b (\Gr, \k) \times D_{\L^+ G}^b (\Gr, \k) \stackrel{\star}{\to} D_{\IW, *}^b (\Gr, \k).\]

    \begin{lemma}\label{lem:iw_action_commutes_with_D_and_T}
        Let $\scrF \in D_{\IW, !}^b (\Gr, \k)$ and $\scrG \in D_{\L^+ G}^b (\Gr, \k)$. There are natural isomorphisms of sheaves in $D_{\IW, *}^b (\Gr, \k)$,
        \begin{enumerate}
            \item $\DD (\scrF) \star \DD (\scrG) \cong \DD (\scrF \star \scrG)$;
            \item $T_!^* (\scrF) \star \scrG \cong T_!^* (\scrF \star \scrG)$.
        \end{enumerate}
    \end{lemma}
    \begin{proof}
        Both statements follow easily from the definitions using base change and the observation that $m$ is proper and $p,q$ are smooth.
    \end{proof}

    Recall that we have an affine bundle $p_{\lambda} : \Gr^{\lambda} \to G/P_{\lambda}$.
    For any $\mu \in w \cdot \lambda$ for $w\in W_f$, we have that $\Gr_{\mu}^{+} = p_{\lambda}^{-1} (X_w^{P_{\lambda}})$.
    In particular, $p_{\lambda}$ restricts to a morphism $p_{\lambda} : \Gr_{\lambda}^{+} \to X_e^{P_{\lambda}}$ realizing $\Gr_{\lambda}^{+}$ as an affine bundle over $X_e^{P_{\lambda}}$ 

    \begin{proposition}\label{prop:reduction_to_flag_variety}
        The functor
        \[p_{\lambda}^* : \Kir_! (X_e^{P_{\lambda}}, \k) \to \Kir_! ( \Gr_{\lambda}^{+}, \k)\]
        is an equivalence of categories. In particular, $\Kir_! ( \Gr_{\lambda}^{+}, \k)$ is nonzero if and only if $\lambda \in \bfXv^{++}$.
    \end{proposition}
    \begin{proof}
        Our $K_{\chi}$-equivariance ensures that 
        \[ p_{\lambda}^* : D_{U_{\chi}^+ \rtimes \G_m}^b (X_e^{P_{\lambda}}, \k) \to D_{K_{\chi} \rtimes \G_m}^b (\Gr_{\lambda}^{+}, \k)\]
        is an equivalence of categories. Clearly after restricting to objects in the Kirillov model, we obtain the desired equivalence of categories.

        Note that $P_{\lambda} = B^{+}$ precisely when $\lambda \in \bfXv^{++}$. As a result, the claimed result on the degeneracy of $\Kir ( \Gr_{\lambda}^{+}, \k)$ when $\lambda \notin \bfXv^{++}$ follows from Proposition \ref{prop:strata_desc_of_Wh_for_flag_varieties}.
    \end{proof}

    For each $\lambda \in \bfXv^{++}$ and $M \in \mod{\k}^{\fg}$, define sheaves
    \[\scrK_{\lambda}^! (M) \coloneq p_{\lambda}^* h_{e!} \overline{\chi}^* (u_* \underline{M}) \qquad\text{and}\qquad \scrK_{\lambda}^* (M) \coloneq p_{\lambda}^* h_{e*} \overline{\chi}^* (u_! \underline{M}).\] 
    Note that $\scrK_{\lambda}^* (M)$ is in $\Kir_* (\Gr_{\lambda}^+, \k)$ since $p_{\lambda}$, $\overline{\chi}$ are smooth and by Proposition \ref{prop:strata_desc_of_Wh_for_flag_varieties}.
    By Propositions \ref{prop:strata_desc_of_Wh_for_flag_varieties} and \ref{prop:reduction_to_flag_variety}, $\Kir_! (\Gr_{\lambda}^{+}, \k)$ (resp. $\Kir_* (\Gr_{\lambda}^{+}, \k)$) is generated by $\scrK_{\lambda}^! (\k)$ (resp. $\scrK_{\lambda}^* (\k)$).
    
    Note that there is a smooth morphism $\chi_{\lambda} : \Gr_{\lambda}^+ \to \A^1$ defined by $\chi_{\lambda} (u \cdot L_{\lambda}) = \chi (u)$ for $u \in I_u^+$.
    We have the following alternate description of the sheaves $\scrK_{\lambda}^! (M)$ and $\scrK_{\lambda}^* (M)$ given by isomorphisms
    \begin{equation}\label{eq:alt_desc_of_FLK}
        \scrK_{\lambda}^! (M) \cong \chi_{\lambda}^* u_* \underline{M} \qquad\text{and}\qquad \scrK_{\lambda}^* (M) \cong \chi_{\lambda}^* u_! \underline{M}.
    \end{equation}

\subsection{Standard and Costandard Objects}

    Iwahori--Whittaker sheaves on the affine Grassmannian have two notable features that make them easier to work with than sheaves in $P_{\L^+ G} (\Gr, \k)$.
    \begin{enumerate}
        \item When $\k$ is a field, $P_{\IW, ?} (\Gr, \k)$ is a highest weight category.
        \item The realization functor $D^b P_{\IW, ?} (\Gr, \k) \to D_{\IW, ?}^b (\Gr, \k)$ is an equivalence of categories.
    \end{enumerate}
    It is also true that $P_{\L^+ G} (\Gr, \k)$ is a highest weight category; however, the argument is rather involved since the $\L^+ G$-orbits are not affine spaces.
    The key difference with Iwahori--Whittaker sheaves is that the highest weight structure is essentially immediate from abstract sheaf theory (see \cite[\S 3.3]{BGS}). 
    The second point is blatantly false when Iwahori--Whittaker sheaves is replaced by $\L^+ G$-equivariant sheaves. 
    The goal of this subsection is to prove these two statements. 

    For $\lambda \in \bfXv^{++}$ and $M \in \mod{\k}^{\fg}$, we set
    \[\Delta_{\lambda}^{\IW, *} (M) \coloneq (j_{\lambda}^+)_! \scrK_{\lambda}^! (\k) [\langle \lambda, 2\rho \rangle] \qquad\text{and}\qquad \nabla_{\lambda}^{\IW, *} (M) \coloneq (j_{\lambda}^{+})_*^{\Kir} \scrK_{\lambda}^! (M) [\langle \lambda, 2\rho \rangle].\]
    These sheaves are called the \emph{standard} and \emph{costandard} Iwahori--Whittaker sheaves in $D_{\IW, !}^b (\Gr, \k)$.
    Note that the standard sheaves are perverse by \cite[Corollaire 4.1.3]{BBD} and the costandards sheaves are perverse by Lemma \ref{lem:costandards_and_perversity}.
    We can analogously define standard and costandard Iwahori--Whittaker sheaves in $D_{\IW, *}^b (\Gr, \k)$,
    \[\Delta_{\lambda}^{\IW, !} (M) \coloneq (j_{\lambda}^+)_!^{\Kir} \scrK_{\lambda}^* (M) [\langle \lambda, 2\rho \rangle] \qquad\text{and}\qquad \nabla_{\lambda}^{\IW, !} (M) \coloneq (j_{\lambda}^{+})_* \scrK_{\lambda}^* (M) [\langle \lambda, 2\rho \rangle].\]
    By Lemma \ref{lem:translation_and_sh_functors} and Example \ref{ex:translation_functor_for_A1}, there are isomorphisms
    \begin{equation}\label{eq:translation_of_std_costds}
        T_!^* \left( \Delta_{\lambda}^{\IW, !} (M)\right) \cong \Delta_{\lambda}^{\IW, *} (M) \qquad\text{and}\qquad T_!^* \left( \nabla_{\lambda}^{\IW, !} (M)\right) \cong \nabla_{\lambda}^{\IW, *} (M).
    \end{equation}
    Likewise, the standard and costandard sheaves are related by Verdier duality by the isomorphisms
    \begin{equation}\label{eq:DD_of_std_costds}
      \DD \left( \Delta_{\lambda}^{\IW, !} (\k)\right) \cong \nabla_{\lambda}^{\IW, *} (\k) \qquad\text{and}\qquad \DD \left( \nabla_{\lambda}^{\IW, !} (\k)\right) \cong \Delta_{\lambda}^{\IW, *} (\k).  
    \end{equation}
    We will almost always drop the $*$- and $!$-superscripts from the notation of the standard and costandard sheaves. One can infer from context whether the objects live in the $!$- or $*$-Kirillov model.
    
    For $\lambda, \mu \in \bfXv^{++}$, a standard argument (cf., \cite[\S 3.3]{BGS}) using Lemma \ref{lem:recollement} and Proposition \ref{prop:reduction_to_flag_variety} shows that
    \begin{equation}\label{eq:hom_std_costd}
        \Hom_{D_{\IW, ?}^b (\Gr, \k)} (\Delta_{\lambda}^{\IW} (M), \nabla_{\mu}^{\IW} (N) [n]) \cong \begin{cases} \Hom_{D^b (\mod{\k}^{\fg})} (M, N [n]) & \text{if }\lambda = \mu, \\ 0 & \text{otherwise.}\end{cases}
    \end{equation}    
    The natural transformation $(j_{\lambda}^{+})_! \to (j_{\lambda}^{+})_*^{\IW}$ induced by adjunction gives a canonical morphism in $\Delta_{\lambda}^{\IW, !} (M) \to \nabla_{\lambda}^{\IW, !} (M)$ which generates the Hom-space as a $\k$-module when $M = \k$.
    We will denote by $\IC_{\lambda}^{\IW, !} (M)$ the image of this generator. 
    Likewise, there is a canonical morphism $\Delta_{\lambda}^{\IW, *} (M) \to \nabla_{\lambda}^{\IW, *} (M)$ induced by adjunction that generates its Hom-space as a $\k$-module when $M = \k$.
    We denote the image of this map by $\IC_{\lambda}^{\IW, *} (M)$. 
    These sheaves are called the \emph{IC sheaves}.
    As with the standard and costandard objects, we will often drop the $!$- or $*$-superscript from the notation of the IC sheaves. 
    When $\k$ is a field, the objects $\{\IC_{\lambda}^{\IW} (\k)\}_{\lambda \in \bfXv^{++}}$ form a complete set of pairwise non-isomorphic  simple objects of the finite-length abelian category $P_{\IW, ?} (\Gr ,\k)$ (see \cite[Proposition 1.4.26]{BBD}).

    Note that since $\zeta$ is minimal in $\bfXv^{++}$ under the partial order $\preceq$, we have that 
    \begin{equation}\label{eq:min_IC_clean}
    \Delta_{\zeta}^{\IW} (M) \cong \IC_{\zeta}^{\IW} (M) \cong \nabla_{\zeta}^{\IW} (M).
    \end{equation}

    \begin{lemma}\label{lem:std_costds_generate}
        $D_{\IW, ?}^b (\Gr, \k)$ is generated by the set of standard sheaves $\{ \Delta_{\lambda}^{\IW} (\k) \}_{\lambda \in \bfXv^{++}}$, or alternatively, the set of costandard sheaves $\{ \nabla_{\lambda}^{\IW} (\k) \}_{\lambda \in \bfXv^{++}}$.
    \end{lemma}
    \begin{proof}
        The lemma is a routine application of recollement (Lemma \ref{lem:recollement}) and Proposition \ref{prop:reduction_to_flag_variety}.
    \end{proof}

    \begin{lemma}\label{lem:t_structure_generation}
        Assume that $\k$ is a field.
        The perverse $t$-structure on $D_{\IW, ?}^b (\Gr, \k)$ is uniquely characterized by each of the following statements:
        \begin{enumerate}
            \item  ${}^p D_{\IW, ?}^b (\Gr, \k)^{\leq 0}$ is generated under extensions by $\Delta_{\lambda}^{\IW} (\k) [n]$ with $\lambda \in \bfXv^{++}$ and $n \geq 0$.
            \item ${}^p D_{\IW, ?}^b (\Gr, \k)^{\geq 0}$ is generated under extensions by $\nabla_{\lambda}^{\IW} (\k) [n]$ with $\lambda \in \bfXv^{++}$ and $n \leq 0$.
        \end{enumerate}
    \end{lemma}
    \begin{proof}
        Let $\text{D}^{\leq 0} \subseteq D_{\IW, ?}^b (\Gr, \k)$ be the smallest full subcategory that is stable under extensions and contains all $\Delta_{\lambda}^{\IW} (\k) [n]$ for $\lambda \in \bfXv^{++}$ and $n \geq 0$.
        Likewise, let $\text{D}^{\geq 0} \subseteq D_{\IW, ?}^b (\Gr, \k)$ be the smallest full subcategory that is stable under extensions and contains all $\nabla_{\lambda}^{\IW} (\k) [n]$ for  $\lambda \in \bfXv^{++}$ and $n \leq 0$.
        It is easy to see that $\text{D}^{\leq 0} \subset {}^p D_{\IW, ?}^b (\Gr, \k)^{\leq 0}$ and  $\text{D}^{\geq 0} \subset {}^p D_{\IW, ?}^b (\Gr, \k)^{\geq 0}$.
        To see that these containments are in fact equalities, it suffices to show that $(\text{D}^{\leq 0}, \text{D}^{\geq 0})$ defines a $t$-structure on $D_{\IW, ?}^b (\Gr, \k)$.
        This follows from \cite[Proposition 1]{Bez01}.
    \end{proof}

    \begin{lemma}\label{lem:h_wt_str}
        Assume that $\k$ is a field.
        The category $P_{\IW} (\Gr, \k)$ is a highest weight category with weight poset $(\bfXv^{++}, \preceq)$, standard objects $\{\Delta_{\lambda}^{\IW} (\k)\}_{\lambda \in \bfXv^{++}}$, and costandard objects $\{\nabla_{\lambda}^{\IW} (\k)\}_{\lambda \in \bfXv^{++}}$.
    \end{lemma}
    \begin{proof}
        The lemma follows from a standard argument using (\ref{eq:hom_std_costd}). See \cite[\S 3.3]{BGS}.
    \end{proof}

    \begin{lemma}\label{lem:rh_perv_t_exact}
        The functor
        \[\scrS : D_{\IW, *}^b (\Gr, \C) \to \Whit^{\dR} (\Gr)\]
        is a $t$-exact equivalence of categories. 
    \end{lemma}
    \begin{proof}
        All the constructions and results so far from this section make sense for $\Whit^{\dR} (\Gr)$ as well (see \cite{ABBGM}).
        It can be easily checked that $\scrS$ takes costandard sheaves to costandard sheaves. By Lemma \ref{lem:std_costds_generate} and Lemma \ref{lem:t_structure_generation}, this implies that $\scrS$ is left $t$-exact and essentially surjective.
        We had already seen that $\scrS$ is fully faithful in \S\ref{subsec:rh}; therefore, $\scrS$ is an equivalence of categories. 
        The right $t$-exactness is then determined by the left $t$-exactness and the fact that $\scrS$ is an equivalence of categories.
    \end{proof}

    \begin{lemma}\label{lem:semisimple_char_0}
        Let $\k$ be a field of characteristic 0. Then $P_{\IW} (\Gr, \k)$ is semisimple. In particular, the natural morphisms
        \[\Delta_{\lambda}^{\IW} \to \IC_{\lambda}^{\IW} \to \nabla_{\lambda}^{\IW}\]
        are all isomorphisms.
    \end{lemma}
    \begin{proof}
        By a standard extension-of-scalars argument, it suffices to consider $\k = \C$.
        The result then follows from Lemma \ref{lem:rh_perv_t_exact} and \cite[Corollary 2.2.3]{ABBGM}.
    \end{proof}

    \begin{remark}
        It should be possible to give a proof of Lemma \ref{lem:semisimple_char_0} which avoids using the Riemann--Hilbert correspondence.
        One strategy would be to prove that $\IC_{\lambda}^{\IW}$ are parity sheaves in the sense of \cite{JMW}.
        Unfortunately, known methods for proving this (for example \cite[Lemma 3.4]{BGMRR}) make use of the decomposition theorem. 
        The decomposition theorem cannot be directly employed since the Fourier--Laumon kernel is not a simple perverse sheaf.  
    \end{remark}

    We say that an object $\scrF$ of $P_{\IW, ?} (\Gr, \k)$ has a \emph{standard filtration} (resp. \emph{costandard filtration}) if it admits a filtration whose subquotients are of the form $\Delta_{\lambda}^{\IW} (\k)$ (resp. $\nabla_{\lambda}^{\IW} (\k)$) for $\lambda \in \bfXv^{++}$.
    When $\k$ is a field, this notion agrees with the usual notion of (co-)standard filtrations in highest weight categories.
    Note that when $\k$ is not a field, we are only allowing $\k$-free (co-)standard objects to appear as subquotients in the filtration.

    For $\lambda \in \bfXv^+$, we will write
    \[\Gr_{\prec \lambda}^+ = \bigcup_{\mu \prec \lambda} \Gr_{\mu}^+ \qquad\text{and}\qquad \Gr_{\preceq \lambda}^+ = \bigcup_{\mu \preceq \lambda} \Gr_{\mu}^+.\]
    For a locally closed union $X$ of finitely many $I^+$-orbits in $\Gr$, we will write $P_{\IW, ?} (X, \k)$ for the perverse heart $\Kir_? (X, \k)$.
    The costandard, standard, and IC objects indexed by $\mu \prec \lambda$ (resp. $\mu \preceq \lambda$) for $\mu \in \bfXv^{++}$ can naturally be regarded as objects in $P_{\IW, ?} (\Gr_{\prec \lambda}^+, \k)$ (resp. $P_{\IW, ?} (\Gr_{\preceq \lambda}^+, \k)$).

    \begin{lemma}\label{lem:gens_for_perv}
        Any object of $P_{\IW, ?} (\Gr, \k)$ is a successive extension of $\IC_{\lambda}^{\IW} (M)$ for some $\lambda \in \bfXv^{++}$ and indecomposable $\k$-modules $M$.
    \end{lemma}
    \begin{proof}
        The proof of the lemma is closely based on \cite[Lemma 2.1.4]{RSW}. Since \emph{loc. cit.} only works with $\k$ being a Dedekind domain and since our theory of IC-sheaves is slightly exotic, we will provide the necessary details.
        We will just prove the case when $? = !$ as the case of $? = *$ is symmetric.
        Since every sheaf $\scrF \in P_{\IW, !} (\Gr, \k)$ is in $P_{\IW, !} (\Gr_{\preceq \lambda}^+, \k)$ for some $\lambda \in \bfXv^{+}$, we can prove the lemma using the following inductive refinement of the lemma.

        (\textbf{IH}) Every object $\scrF$ in $P_{\IW, !} (\Gr_{\preceq \lambda}^+, \k)$ is a successive extension of $\IC_{\mu}^{\IW} (M)$ for some $\mu \in \bfXv^{++}$ with $\mu \preceq \lambda$ and indecomposable $\k$-modules $M$.
        
        The base case when $\lambda = 0$ follows from Proposition \ref{prop:reduction_to_flag_variety}.
        Let $\lambda \in \bfXv^+$. Write $i : \Gr_{\prec \lambda}^+ \hookrightarrow \Gr_{\preceq \lambda}^+$ and $j : \Gr_{\lambda}^+ \hookrightarrow \Gr_{\preceq \lambda}^+$ for the inclusion maps.
        The functor $i_* : P_{\IW, !} (\Gr_{\prec \lambda}^+, \k) \to P_{\IW, !} (\Gr_{\preceq \lambda}^+, \k)$ is fully faithful, so we may assume by induction that any object in $P_{\IW, !} (\Gr_{\prec \lambda}^+, \k)$ is a successive extension of $\IC_{\mu}^{\IW} (M)$ for some $\mu \prec \lambda$ and indecomposable $\k$-modules $M$.
        Let $\scrF$ be an arbitrary object in $P_{\IW, !} (\Gr, \k)$.
       By \cite[Proposition 1.4.17 (ii)]{BBD}, the natural morphism $\scrF \to i_* {}^p H^0 (i^* \scrF)$ induced by adjunction is surjective. We denote its kernel by $\scrK$, and consider the short exact sequence of perverse sheaves
       \[0 \to \scrK \to \scrF \to i_* {}^p H^0 (i^* \scrF) \to 0.\]
       By induction, it then suffices to prove that $\scrK$ is a successive extension of $\IC_{\mu}^{\IW} (M)$ for some $\mu \preceq \lambda$ and indecomposable $\k$-modules $M$.
       Note that $\scrK$ has no non-zero quotient supported on $\Gr_{\lambda}^+$. By Lemma \ref{lem:translation_and_sh_functors} and \cite[Proposition 1.4.17 (ii)]{BBD}, the natural morphism $i_* {}^p H^0 (i_{\Kir}^! \scrK) \to \scrK$ is injective. We denote its cokernel by $\scrQ$.
       There is then a short exact sequence
       \[0 \to i_* {}^p H^0 (i_{\Kir}^! \scrK) \to \scrK \to \scrQ \to 0.\]
       By induction, it suffices to prove that $\scrQ$ is a successive extension of $\IC_{\mu}^{\IW} (M)$ for some $\mu \preceq \lambda$ and indecomposable $\k$-modules $M$.
       Moreover, $\scrQ$ has no non-trivial subobject or quotient supported on $\Gr_{\lambda}^+$. 
       Define the intermediate-extension functor
       \[j_{!*}^{\Kir} : P_{\IW, !} (\Gr_{\lambda}^+, \k) \to P_{\IW, !} (\Gr_{\preceq \lambda}^+, \k), \qquad \scrF \mapsto \im (j_! \scrF \to j_*^{\Kir} \scrF).\]
       Note that $\IC_{\lambda}^{\IW} (M) = j_{!*}^{\Kir} \scrK_{\lambda}^! (M) [\langle \lambda, 2\rho \rangle]$.
       The characterization of IC sheaves given in \cite[Corollaire 1.4.25]{BBD} can be easily translated to our setting, and in particular, we conclude that there is an isomorphism $\scrQ \cong j_{!*}^{\Kir} j^* \scrQ$. 
       By Proposition \ref{prop:reduction_to_flag_variety} and Proposition \ref{prop:strata_desc_of_Wh_for_flag_varieties}, we must have that $(j_{\lambda}^+)^* \scrQ$ is either 0 (when $\lambda \notin \bfXv^{++}$) or of the form $\scrK_{\lambda}^! (M) [\langle \lambda, 2\rho \rangle]$ where $M \in \mod{\k}^{\fg}$.
       Of course, $M$ is a direct sum of indecomposable $\k$-modules, and hence, the lemma follows.
    \end{proof}

    \begin{lemma}\label{lem:existence_of_projs_for_IC}
        For all $\lambda \in \bfXv^{++}$, there exists a projective object $\scrP$ in $P_{\IW, ?} (\Gr, \k)$ which admits a standard filtration and a surjection $\scrP \twoheadrightarrow \IC_{\lambda}^{\IW} (\k)$.
    \end{lemma}
    \begin{proof}
       The proof of the lemma is closely based on \cite[Proposition 2.3.1]{RSW}. By Lemma \ref{lem:translation_and_sh_functors} and Proposition \ref{prop:translation_functors_are_equivs}, it suffices to just prove the $!$-Kirillov variant. 
       As in Lemma \ref{lem:gens_for_perv}, we will prove the lemma using the following inductive refinement.

        (\textbf{IH}) For all $\mu \preceq \lambda$ with $\mu \in \bfXv^{++}$, there exists a projective object $\scrP_{\mu} \in P_{\IW, !} (\Gr_{\preceq \lambda}^+, \k)$ which admits a standard filtration and a surjection $\scrP_{\mu} \to \IC_{\mu}^{\IW} (\k)$.

        The base case when $\lambda = 0$ is obvious. Let $\lambda \in \bfXv^+$ and assume the inductive hypothesis holds for all $\lambda' \prec \lambda$.
        We define $\scrP_{\lambda} = \Delta_{\lambda}^{\IW} (\k)$ when $\lambda \in \bfXv^{++}$. 
        It is clear from its definition that $\scrP_{\lambda}$ has a standard filtration and admits a surjective morphism $\scrP_{\lambda} \twoheadrightarrow \IC_{\lambda}^{\IW} (\k)$.
        We can first check that $\scrP_{\lambda}$ is projective. For any $\scrF \in P_{\IW, !} (\Gr_{\preceq \lambda}^+, \k)$, there is an isomorphism
        \[\Ext^1 (\scrP_{\lambda}, \scrF) \cong \Hom_{\Kir_! (\Gr_{\lambda}^+, \k)} (\scrK_{\lambda}^! [\langle \lambda, 2\rho \rangle ], (j_{\lambda}^+)^* \scrF [1])\]
        The right-hand side vanishes by Proposition \ref{prop:reduction_to_flag_variety} and Proposition \ref{prop:strata_desc_of_Wh_for_flag_varieties}.

        Now take $\mu \prec \lambda$. By induction, we have a projective object $\scrP_{\mu}^{\prec \lambda}$ in $P_{\IW, !} (\Gr_{\prec \lambda}^+, \k)$ which admits a standard filtration and a surjective morphism to $\Delta_{\mu}^{\IW} (\k)$.
        It remains to check that $\scrP_{\mu}^{\prec \lambda}$ can be enlarged to a projective object in  $P_{\IW, !} (\Gr_{\preceq \lambda}^+, \k)$ which admits a standard filtration and a surjective morphism to $\Delta_{\mu}^{\IW} (\k)$.
        To do so, consider $E = \Ext^1 (\scrP_{\mu}^{\prec \lambda}, \Delta_{\lambda}^{\IW} (\k))$ and let $E_{\textnormal{free}}$ be a finitely generated free $\k$-module with a surjection $E_{\textnormal{free}} \twoheadrightarrow E$.
        Let $E_{\textnormal{free}}^*$ denote the dual $\k$-module of $E_{\textnormal{free}}$.
        The sequence of canonical morphisms
        \[\k \to E_{\textnormal{free}}^* \otimes_{\k} E_{\textnormal{free}} \to E_{\textnormal{free}}^* \otimes_{\k} E \cong \Ext^1 (\scrP_{\mu}^{\prec \lambda}, E_{\textnormal{free}}^* \otimes_{\k} \Delta_{\lambda}^{\IW} (\k))\]
        gives rise to a distinguished extension
        \[0 \to E_{\textnormal{free}}^* \otimes_{\k} \Delta_{\lambda}^{\IW} (\k) \to \scrP_{\mu} \to \scrP_{\mu}^{\prec \lambda} \to 0.\]
        By induction $\scrP_{\mu}$ has a standard filtration. Moreover, it admits a surjection to $\IC_{\mu}^{\IW} (\k)$ via the composition $\scrP_{\mu} \twoheadrightarrow \scrP_{\mu}^{\prec \lambda} \twoheadrightarrow \IC_{\mu}^{\IW} (\k)$.
        To complete the proof of the lemma, it remains to check that $\scrP_{\mu}$ is projective. 

        \emph{Step 1.} $\Ext^1 (\scrP_{\mu}, \scrF) = 0$ for $\scrF \in P_{\IW} (\Gr_{\prec \lambda}^+, \k)$. From the definition of $\scrP_{\mu}$, there is a long exact sequence
        \[\ldots \to \Ext^1 (\scrP_{\mu}^{\prec \lambda}, \scrF) \to \Ext^1 (\scrP_{\mu}, \scrF) \to \Ext^1 (E_{\textnormal{free}}^* \otimes_{\k} \Delta_{\lambda}^{\IW} (\k), \scrF ) \to \ldots.\]
        The desired vanishing then follows from $\Delta_{\lambda}^{\IW} (\k)$ being projective in $P_{\IW} (\Gr_{\preceq \lambda}^+, \k)$ and $\scrP_{\mu}^{\prec \lambda}$ being projective in $P_{\IW} (\Gr_{\prec \lambda}^+, \k)$.

        \emph{Step 2.} $\Ext^1 (\scrP_{\mu}, \Delta_{\lambda}^{\IW} (\k)) = 0$. There is a long exact sequence
        \begin{align*}
            \ldots \to \Hom (E_{\textnormal{free}^* \otimes_{\k} \Delta_{\lambda}^{\IW} (\k), \Delta_{\lambda}^{\IW} (\k)}) \to \Ext^1& (\scrP_{\mu}^{\prec \lambda}, \Delta_{\lambda}^{\IW} (\k)) \to \Ext^1 (\scrP_{\mu}, \Delta_{\lambda}^{\IW} (\k)) \\
            &\to \Ext^1 (E_{\textnormal{free}^*} \otimes_{\k} \Delta_{\lambda}^{\IW} (\k), \Delta_{\lambda}^{\IW} (\k)) \to \ldots
        \end{align*}
        By adjunction, one has that $\Ext^1 (\Delta_{\lambda}^{\IW} (\k), \Delta_{\lambda}^{\IW} (\k)) = 0$, so it suffices to show that the first map is surjective.
        However, this map can be explicitly described under the canonical isomorphism 
        \[\Hom (E_{\textnormal{free}^* \otimes_{\k} \Delta_{\lambda}^{\IW} (\k), \Delta_{\lambda}^{\IW} (\k)}) \cong E_{\textnormal{free}} \otimes_{\k} (\Delta_{\lambda}^{\IW} (\k), \Delta_{\lambda}^{\IW} (\k)) \cong  E_{\textnormal{free}}\]
        as the map $E_{\textnormal{free}} \to E = \Ext^1 (\scrP_{\mu}^{\prec \lambda}, \Delta_{\lambda}^{\IW} (\k))$, which is surjective by construction.

        \emph{Step 3.} $\Hom (\scrP_{\mu}, \scrF [2]) = 0$ for all $\scrF \in P_{\IW} (\Gr_{\preceq \lambda}^+, \k)$. By Lemma \ref{lem:gens_for_perv}, it suffices to show that $\Hom (\scrP_{\mu}, \IC_{\nu}^{\IW} (M)) = 0$ for all $\nu \preceq \lambda$ and $M$ finitely generated $\k$-modules.
        Let $\scrQ$ denote the cokernel of the canonical inclusion $\IC_{\nu}^{\IW} (M) \to \nabla_{\mu}^{\IW} (M)$. We then have a long exact sequence
        \[\ldots \to \Ext^1 (\scrP_{\mu}, \scrQ) \to \Hom (\scrP_{\mu}, \IC_{\nu}^{\IW} (M) [2]) \to \Hom (\scrP_{\mu}, \nabla_{\nu}^{\IW} (M) [2]) \to \ldots.\]
        The first term vanishes from Step 1. Since $\scrP_{\mu}$ has a standard filtration, the third term will vanish by (\ref{eq:hom_std_costd}).

        \emph{Step 4.} $\Ext^1 (\scrP_{\mu}, \IC_{\lambda}^{\IW} (M)) = 0$ for all $M \in \mod{\k}^{\fg}$. 
        We can find a finitely generated free $\k$-module $M_{\textnormal{free}}$ with a surjection $M_{\textnormal{free}} \twoheadrightarrow M$.
It can be checked as in \cite[Lemma 3.3.5]{Ac} that the functor $\IC_{\lambda}^{\IW} (-) : \mod{\k}^{\fg} \to P_{\IW, !} (\Gr, \k)$ taking $N$ to $\IC_{\lambda}^{\IW} (N)$ takes surjective maps to surjective maps.
In particular, there is a surjective map $M_{\textnormal{free}} \otimes_{\k} \Delta_{\lambda}^{\IW} (\k) \to \IC_{\lambda}^{\IW} (\k)$. We denote its kernel by $\scrK$.
We can now consider the long exact sequence
\[\ldots \to \Ext^1 (\scrP_{\mu}, M_{\textnormal{free}} \otimes_{\k} \Delta_{\lambda}^{\IW} (\k)) \to \Ext^1 (\scrP_{\mu}, \IC_{\lambda}^{\IW} (M)) \to \Hom (\scrP_{\mu}, \scrK [2]) \to \ldots.\]
The first and third terms are zero by Steps 2 and 3 respectively.

Finally, we have seen from Steps 1 and 4 that $\Ext^1 (\scrP_{\mu}, \IC_{\nu}^{\IW} (M)) = 0$ for all $\nu \preceq \lambda$ and $M \in \mod{\k}^{\fg}$. We conclude from Lemma \ref{lem:gens_for_perv} that $\scrP_{\mu}$ is projective.
    \end{proof}

    \begin{lemma}\label{lem:enough_projectives}
        For any $\scrF \in P_{\IW, ?} (\Gr, \k)$, there exists a projective object $\scrP$ in $P_{\IW, ?} (\Gr, \k)$ which admits a standard filtration and a surjection $\scrP \twoheadrightarrow \scrF$.
        In particular, there are enough projective objects in $P_{\IW, ?} (\Gr, \k)$.
    \end{lemma}
    \begin{proof}
        By Lemma \ref{lem:gens_for_perv}, we may assume that $\scrF = \IC_{\lambda}^{\IW} (M)$ where $\lambda \in \bfXv^{++}$ and $M$ is an indecomposable $\k$-module.
        When $M = \k$, the desired projective object can be constructed by Lemma \ref{lem:existence_of_projs_for_IC}.
        For a general $M$, we can find a finitely generated free $\k$-module $M_{\textnormal{free}}$ along with a surjection $M_{\textnormal{free}} \twoheadrightarrow M$. 
        As in the proof of Step 4 in Lemma \ref{lem:existence_of_projs_for_IC}, we have a surjection $M_{\textnormal{free}} \otimes_{\k} \IC_{\lambda}^{\IW} (\k) \twoheadrightarrow \IC_{\lambda}^{\IW} (M)$.
        We are then done by the $M = \k$ case.
    \end{proof}

    \begin{proposition}\label{prop:realization_is_equivalence}
        The realization functor
        \[D^b P_{\IW, ?} (\Gr, \k) \to D_{\IW, ?}^b (\Gr, \k)\]
        is an equivalence of categories.
    \end{proposition}
    \begin{proof}
        Our proof will closely follow \cite[Corollary 2.3.4]{RSW}. By a standard homological argument, it suffices to prove that the natural morphism
        \begin{equation}\label{eq:realization_is_equivalence_1}
            \Ext_{P_{\IW, ?} (\Gr, \k)}^i (\scrP, \scrF) \stackrel{\sim}{\to} \Hom_{D_{\IW, ?}^b (\Gr, \k)} (\scrP, \scrF [i])
        \end{equation}
        is an isomorphism for all $i \in \Z$, $\scrP$ projective in $P_{\IW, ?} (\Gr, \k)$ with a standard filtration, and $\scrF \in P_{\IW, ?} (\Gr, \k)$. 
        If $i \leq 0$, the isomorphism is obvious. We will prove by induction that when $i > 0$ then both sides of (\ref{eq:realization_is_equivalence_1}) vanish.
        For $i = 1$, the left-hand side vanishes since $\scrP$ is projective. On the other hand, (\ref{eq:realization_is_equivalence_1}) is an isomorphism by \cite[Remarque 3.1.17]{BBD}, so the right-hand side also vanishes.
        Assume the claim is known for some $i \geq 1$.
        By Lemma \ref{lem:gens_for_perv}, it suffices to consider the case when $\scrF = \IC_{\lambda}^{\IW} (M)$ where $\lambda \in \bfXv^{++}$ and $M$ is an indecomposable $\k$-module.
        We can consider the short exact sequence
        \[0 \to \IC_{\lambda}^{\IW} (M) \to \nabla_{\lambda}^{\IW} (M) \to \scrQ \to 0\]
        where $\scrQ$ is the cokernel of the first map. We can apply $R\Hom_{D_{\IW, ?}^b (\Gr, \k)} (\scrP, -)$ and taking cohomology to produce a long exact sequence
        \[\ldots \to \Hom (\scrP, \scrQ [i]) \to \Hom(\scrP, \IC_{\lambda}^{\IW} (M) [i+1]) \to \Hom (\scrP, \nabla_{\lambda}^{\IW} (M) [i+1]) \to \ldots.\]
        The first term vanishes by induction. The last term vanishes because $\scrP$ has a standard filtration using (\ref{eq:hom_std_costd}).
        We deduce that the middle term must also vanish.
    \end{proof}

    \section{Geometric Casselman--Shalika Equivalence}
    
    We are now ready to state the main result of the paper. The proof will occupy the remainder of the section. 

    \begin{theorem}\label{thm:cs}
        The functor 
        \[ \Phi_? : D_{\L^+ G}^b (\Gr, \k) \to D_{\IW, ?}^b (\Gr, \k)\] 
        defined by $\Phi_? (\scrF) \coloneq \Delta_{\zeta}^{\IW, ?} (\k) \star \scrF$ is perverse $t$-exact, and restricts to an equivalence of abelian categories
        \[P_{\L^+ G} (\Gr, \k) \to P_{\IW, ?} (\Gr, \k).\]
    \end{theorem}

     \subsection{Exactness and Explicit Description of \texorpdfstring{$\Phi_?$}{Phi}}

    \begin{lemma}\label{lem:Phi_t_exact}
        The functor $\Phi_?$ is perverse $t$-exact.
    \end{lemma}
    \begin{proof}
        There are two preliminary observations that we can make. First, if $\k \to \k'$ is a ring homomorphism, then there is a natural isomorphism
        \begin{equation}\label{eq:Phi_t_exact_1}
            \k' \otimes_{\k}^L \Phi_? (-) \cong \Phi_? (\k' \otimes_{\k}^L -).
        \end{equation}
        Second, we note that by Lemma \ref{lem:H_conv_t_exactness}, when $\k$ is a field, the functor $\Phi_?$ is perverse $t$-exact.

        We will first argue that $\Phi_?$ is right $t$-exact for general coefficients. Let $\scrF \in P_{\L^+G} (\Gr, \k)$.
        Let $n$ be the largest integer such that ${}^p H^n (\Phi_? (\scrF)) \neq 0$.
        By \cite[Lemma 3.2.6]{Ac}, we can find a ring homomorphism $\k \to \F$ with $\F$ a field such that ${}^p H^n (\F \otimes_{\k}^L \Phi_? (\scrF)) \neq 0$ and $\F \otimes_{\k}^L  \Phi_? (\scrF) \in {}^p D_{\IW}^b (\Gr, \F)^{\leq n}$.
        Moreover, the natural isomorphism (\ref{eq:Phi_t_exact_1}) shows that these conditions also hold for $\Phi_? (\F \otimes_{\k}^L \scrF)$.
        On the other hand, since $\F$ is a field and extension-of-scalars is right $t$-exact, we must have that $\Phi_? (\F \otimes_{\k}^L \scrF) \in {}^p D_{\IW}^b (\Gr, \F)^{\leq 0}$. We deduce that $n \leq 0$, and hence that $\Phi_? (\scrF) \in {}^p D_{\IW}^b (\Gr, \k)^{\leq 0}$. 

        We can now check that $\Phi_?$ is left $t$-exact. Again we take $\scrF \in P_{\L^+G} (\Gr, \k)$.
        Let $n$ be the smallest integer such that ${}^p H^n (\Phi_? (\scrF)) \neq 0$.
        By \cite[Lemma 3.2.7]{Ac}, we can find a prime ideal $\fr{p} \subset \k$ such that ${}^p H^{n - \textnormal{ht} (\fr{p})} (\kappa_{\fr{p}} \otimes_{\k}^L \Phi_? (\scrF)) \neq 0$ and $\kappa_{\fr{p}} \otimes_{\k}^L  \Phi_? (\scrF) \in {}^p D_{\IW}^b (\Gr, \kappa_{\fr{p}})^{\geq n - \textnormal{ht} (\fr{p})}$, where $\kappa_{\fr{p}}$ is the residue field of $\fr{p}$ and $\textnormal{ht} (\fr{p})$ is the height of $\fr{p}$.
        Moreover, (\ref{eq:Phi_t_exact_1}) shows that these conditions also hold for $\Phi_? (\kappa_{\fr{p}} \otimes_{\k}^L \scrF)$.
        On the other hand, again by \cite[Lemma 3.2.7]{Ac} and the field case, $\Phi_? (\kappa_{\fr{p}} \otimes_{\k}^L \scrF) \in {}^p D_{\IW}^b (\Gr, \kappa_{\fr{p}})^{\geq - \textnormal{ht} (\fr{p})}$.
        This shows that $n \geq 0$, and hence that $\Phi_? (\scrF) \in {}^p D_{\IW}^b (\Gr, \k)^{\geq 0}$.
    \end{proof}

    Recall that for $\lambda \in \bfXv^{++}$, there is a morphism $\chi_{\lambda} : \Gr_{\lambda}^+ \to \G_a$ defined by $\chi_{\lambda} (u \cdot L_{\lambda}) = \chi (u)$ for $u \in I_u^+$.
    From (\ref{eq:alt_desc_of_FLK}), we have isomorphisms $\scrK_{\zeta}^! (\k) \cong \chi_{\zeta}^* u_* \k$ and $\scrK_{\zeta}^* (\k) \cong \chi_{\zeta}^* u_! \k$.

    \begin{lemma}\label{lem:desc_of_Phi}
        Let $J$ be the stabilizer of the point $L_{\zeta}$ in $I_u^+$ and write $\phi_{\zeta} : \Gr \to \Gr$ for the automorphism $x \mapsto z^{\zeta} \cdot x$.
        Then there are natural isomorphisms of functors
        \[\Phi_! \cong \Av_{\Kir}^! \circ \Av_{J \rtimes \G_m, *}^{K_{\chi} \rtimes \G_m} \circ \phi_{\zeta *} \circ \For^{\L^+ G}_{z^{-\zeta} J z^{\zeta} \rtimes \G_m} [\langle \zeta, 2\rho \rangle - 1].\]        
        \[\Phi_* \cong \Av_{\Kir}^* \circ \Av_{J \rtimes \G_m, !}^{K_{\chi} \rtimes \G_m} \circ \phi_{\zeta *} \circ \For^{\L^+ G}_{z^{-\zeta} J z^{\zeta} \rtimes \G_m} [-\langle \zeta, 2\rho \rangle + 1].\]        
    \end{lemma}
    \begin{proof}
        We will just prove the description of $\Phi_!$. The description of $\Phi_*$ can be obtained by Verdier duality.
        Let $i : \{L_{\zeta}\} \hookrightarrow \Gr$ denote the inclusion map. 
        The sheaf $\delta_{L_{\zeta}} \coloneq i_* \k$ is naturally $J \rtimes \G_m$-equivariant since $L_{\zeta}$ is stabilized by $J \rtimes \G_m$. 
        We claim that there is an isomorphism
        \begin{equation}\label{eq:desc_of_Phi_1}
            \Delta_{\zeta}^{\IW} \cong \Av_{\Kir}^!  \Av_{J \rtimes \G_m, *}^{K_{\chi} \rtimes \G_m} \delta_{L_{\zeta}} [\langle \zeta, 2\rho \rangle - 1].
        \end{equation}
        We can compute
        \begin{align*}
            R\Hom_{D_{\IW, !}^b (\Gr, \k)} (\Delta_{\zeta}^{\IW}, \Av_{\Kir}^!  \Av_{J \rtimes \G_m, *}^{K_{\chi} \rtimes \G_m} \delta_{L_{\zeta}} [\langle \zeta, 2\rho \rangle - 1]) &\cong R\Hom_{D_{J \rtimes \G_m}^b (\Gr, \k)} (\Delta_{\zeta}^{\IW}, \delta_{L_{\zeta}} [\langle \zeta, 2\rho \rangle - 1]) \\
            &\cong R\Hom_{D_{J \rtimes \G_m}^b (\pt, \k)} (i^* (j_{\zeta}^+)_! \chi_{\zeta}^* u_* \k, \k [-1] ) \\
            &\cong R\Hom_{D_{J \rtimes \G_m}^b (\pt, \k)} (i^* (j_{\zeta}^+)_! \chi_{\zeta}^* u_* \k, \k [-1] ) \\ 
            &\cong R\Hom_{D_{\G_m}^b (\pt, \k)} (s^* u_* \k, \k [-1]) \\
            &\cong R\Hom_{D_{\G_m}^b (\pt, \k)} (\Av_!^{\G_m} \k [-1], \k [-1]) \\
            &\cong R\Hom_{D_{\cons}^b (\pt, \k)} (\k [-1], \k [-1]) \\
            &\cong \k.
        \end{align*}
        The isomorphism in (\ref{eq:desc_of_Phi_1}) then follows from Proposition \ref{prop:reduction_to_flag_variety} and Proposition \ref{prop:strata_desc_of_Wh_for_flag_varieties}.

        Let $\scrF \in D_{\L^+ G}^b (\Gr, \k)$. By Lemma \ref{lem:av_and_convolution} and (\ref{eq:desc_of_Phi_1}), we have natural isomorphisms
        \[\Phi_! (\scrF) \cong \Av_{\Kir}^! \circ \Av_{J \rtimes \G_m, *}^{K_{\chi} \rtimes \G_m} (\delta_{L_{\zeta}} \star \scrF) [\langle \zeta, 2\rho \rangle - 1].\]
        To complete the lemma, one can observe from the definitions that $\delta_{L_{\zeta}} \star \scrF \cong \phi_{\zeta *} \circ \For^{\L^+ G}_{z^{-\zeta} J z^{\zeta} \rtimes \G_m}$.
    \end{proof}

    \subsection{Reducing Equivariant Cohomology}

    Recall that there are functors of $\G_m$-invariants and $\G_m$-coinvariants,
    \[\Inv_{\G_m *} : D_{\G_m}^b (\pt, \k) \to D^+ (\mod{\k}^{\fg}) \qquad\text{and}\qquad \Inv_{\G_m !} : D_{\G_m}^b (\pt, \k) \to D^- (\mod{\k}^{\fg}).\]
    These functors (or rather their truncated versions) are thoroughly treated in \cite[Chapter 6]{Ac}.
    To define them, we will instead use the language of algebraic stacks.
    Under the identification $D_{\cons}^+ (\pt, \k) \cong D^+ (\mod{\k}^{\fg})$, $\Inv_{\G_m *}$ corresponds to the $*$-pushforward along the map $q : [\pt/\G_m] \to \pt$. 
    Note that the presence of non-bounded complexes of $\k$-modules occurs since $q$ is not representable.
    Likewise, under the identification $D_{\cons}^+ (\pt, \k) \cong D^+ (\mod{\k}^{\fg})$, the functor $\Inv_{\G_m !}$ corresponds to $q_! [-2]$.

    Let $X$ be a $\G_m$-variety and write $a : X \to \pt$ for the canonical map. We define functors
    \[R\Gamma_{\G_m} \coloneq \Inv_{\G_m *} \circ a_* : D_{\G_m}^b (X, \k) \to D^+ (\mod{\k}^{\fg}),\] 
    \[R\Gamma_{c, \G_m} \coloneq \Inv_{\G_m !} \circ a_! : D_{\G_m}^b (X, \k) \to D^- (\mod{\k}^{\fg}). \] 
    The notation here is a non-standard, since the cohomology of $R\Gamma_{c, \G_m}$ does not compute the $\G_m$-equivariant cohomology with compact support.
    Instead, the cohomology of $R\Gamma_{c, \G_m}$ computes the cohomology with compact support of sheaves on the quotient stack $[\G_m \backslash X]$ (up to a shift). 
    We favor the above interpretation since it behaves well with Verdier duality and $\G_m$-averaging. 
    Namely, there are isomorphisms 
    \[ \DD \circ R\Gamma_{\G_m} \cong R\Gamma_{c, \G_m} \circ \DD,\]
    \[ R\Gamma \cong R\Gamma_{\G_m} \circ \Av_*^{\G_m}, \qquad\text{and}\qquad R\Gamma_c \cong R\Gamma_{c, \G_m} \circ \Av_!^{\G_m}\]

    \begin{lemma}\label{lem:G_m_equiv_cohomology_comparison}
        Let $X$ be a $\G_m$-variety and $\scrF \in D_{\G_m}^b (X, \k)$. 
        There is a distinguished triangle,
        \[R\Gamma_c (\For^{\G_m} (\scrF)) [-2] \to R\Gamma_{c, \G_m} (\scrF) [-2] \to R\Gamma_{c, \G_m} (\scrF) \to.\]
    \end{lemma}
    \begin{proof}
        We will prove the lemma using stacky language.
        Let $\pi : X \to [\G_m \backslash X]$ denote the quotient map. Note that $\pi$ makes $X$ into a $\G_m$-torsor over $[\G_m \backslash X]$.
        Consider the line bundle $Y = \A^1 \times^{\G_m} X$ on $[\G_m \backslash X]$ corresponding to $X$.  
        Let $q : Y \to [\G_m \backslash X]$ denote the corresponding projection map, and write $i : [\G_m \backslash X] \hookrightarrow Y$ for the 0-section of the line bundle.
        The complement of $[\G_m \backslash X]$ in $Y$ is canonically isomorphic to $X$, and we can write $j : X \hookrightarrow Y$ for its inclusion. 
        Note that $q \circ j = \pi$. Let $\scrF \in D_{\G_m}^b (X, \k)$, and consider the excision triangle,
        \[j_! \pi^* \scrF \to q^* \scrF \to i_* \scrF \to.\]
        By the projection formula, $q_! q^* \scrF \cong \scrF [-2]$, so after applying $q_!$ to the above distinguished triangle we get
        \[\pi_! \pi^* \scrF \to \scrF [-2] \to \scrF \to.\]
        Note that $\pi^*$ corresponds to the forgetful functor $\For^{\G_m}$ and $\pi_!$ corresponds to the shifted $\G_m$-averaging functor $\Av_!^{\G_m} [-2]$.
        Finally, we apply $R\Gamma_{c, \G_m}$ to deduce the desired distinguished triangle.
    \end{proof}

    \subsection{Image of Standard and Costandards}

    Consider the composition
    \[\chi_{\L U^+} : \L U^+ \stackrel{\L\chi}{\longrightarrow} \L \G_a \to \G_a,\]
    where the second map is given by taking residues, $\sum_{i \in \Z} a_i z^i \mapsto a_{-1}$.
    For $\mu \in \bfXv^+$, we will consider the semi-infinite orbit 
    \[ S_{\mu} \coloneq \L U^+ \cdot L_{\mu}.\]
    We can define a $\G_m$-equivariant morphism $\chi_{\mu}' : S_{\mu} \to \G_a$ such that $\chi_{\mu}' (u \cdot L_{\mu}) = \chi_{\L U^+} (u)$ for all $u \in \L U^+$.

    Consider the subvariety $z^{- \zeta} \Gr_{\zeta + \mu}^+ \subseteq \Gr$. The morphism $\chi_{\zeta + \mu} : \Gr_{\zeta + \mu}^+ \to \G_a$ can be translated to a morphism $\chi_{\mu}^{\zeta} : z^{- \zeta} \Gr_{\zeta + \mu}^+ \to \G_a$. Explicitly, $\chi_{\mu}^{\zeta} (z^{-\zeta} \cdot x) = \chi_{\zeta + \mu} (x)$ for $x \in \Gr_{\zeta + \mu}^+$.
    As in the proof of \cite[Lemma 3.10]{BGMRR}, the variety $z^{- \zeta} \Gr_{\zeta + \mu}^+$ is contained in $S_{\mu}$. Moreover, $\chi_{\mu}^{\zeta}$ is the restriction of $\chi_{\mu}'$ to $z^{- \zeta} \Gr_{\zeta + \mu}^+$.

    For $\lambda, \mu \in \bfXv^{+}$, we define a subvariety $Y_{\lambda, \mu} \coloneq \Gr^{\lambda} \cap (z^{-\zeta} \Gr_{\zeta + \mu}^+) \subseteq z^{-\zeta} \Gr_{\zeta + \mu}^+$.
    Note that $Y_{\lambda, \mu}$ is a subvariety of the Mirković--Vilonen cycle $\Gr^{\lambda} \cap S_{\mu}$. We write $h_{\lambda, \mu} : Y_{\lambda,\mu} \hookrightarrow z^{-\zeta} \Gr_{\zeta + \mu}^+$ for the inclusion map.

    \begin{lemma}\label{lem:hom_and_cohomology_std_costd}
        Assume that $\k$ is a field. Let $\lambda, \mu \in \bfXv^+$ with $\lambda \neq \mu$ and $n \geq 0$.
        There is an isomorphism
        \[\Hom_{D_{\IW, *}^b (\Gr, \k)} (\Phi_* (\scrJ_! (\lambda, \k)), \nabla_{\zeta + \mu}^{\IW} [-n]) \cong H^{\langle \lambda + \mu, 2\rho \rangle + n + 1} \left( R\Gamma_{c, \G_m} (((\chi_{\mu}^{\zeta})^* u_* \k)\vert_{Y_{\lambda, \mu}} )\right)^*.\]
        In particular, the $t$-exactness for $\Phi_*$ ensures that
        \[\tau^{> \langle \lambda + \mu, 2\rho \rangle + 1} R\Gamma_{c, \G_m} (((\chi_{\mu}^{\zeta})^* u_* \k)\vert_{Y_{\lambda, \mu}}) = 0.\]
    \end{lemma}
    \begin{proof}
        Note that $j_{\lambda !}$ is right $t$-exact, so $j_{\lambda !} \uk [\langle \lambda, 2\rho \rangle] \in {}^p D_{\L^+ G}^b (\Gr, \k)^{\leq 0}$.
        We can then consider the truncation distinguished triangle
        \[{}^p \tau^{\leq -1} j_{\lambda !} \uk [\langle \lambda, 2\rho \rangle] \to j_{\lambda !} \uk [\langle \lambda, 2\rho \rangle] \to \scrJ_! (\lambda, \k) \to.\]
        We can apply $R\Hom (\Phi_* (-), \nabla_{\zeta + \mu}^{\IW})$ and take cohomology to deduce that there is an isomorphism
        \[\Hom_{D_{\IW, *}^b (\Gr, \k)} (\Phi_* (\scrJ_! (\lambda, \k)), \nabla_{\zeta + \mu}^{\IW} [-n]) \cong \Hom_{D_{\IW, *}^b (\Gr, \k)} (\Phi_* (j_{\lambda !} \uk [\langle \lambda, 2\rho \rangle]), \nabla_{\zeta + \mu}^{\IW} [-n]). \]
         By Lemma \ref{lem:desc_of_Phi}, we have isomorphisms
         \begin{align*}
         \Hom_{D_{\IW, *}^b (\Gr, \k)} (\Phi_* (j_{\lambda !} \uk [\langle \lambda, 2\rho \rangle]), &\nabla_{\zeta + \mu}^{\IW} [-n]) \cong \Hom_{D_{\G_m}^b (\Gr, \k)} ( j_{\lambda !} \uk [\langle \lambda - \zeta, 2\rho \rangle + 1],  \phi_{\zeta *}^{-1} \nabla_{\zeta + \mu}^{\IW} [-n])          \\
         &\cong \Hom_{D_{\G_m}^b (\Gr^{\lambda}, \k)} (\uk_{\Gr^{\lambda}},  j_{\lambda}^! \phi_{\zeta *}^{-1} \nabla_{\zeta + \mu}^{\IW} [\langle \zeta - \lambda, 2\rho \rangle - n - 1]) \\
         &\cong H_{\G_m}^{\langle \zeta - \lambda, 2\rho \rangle - n - 1} (\Gr^{\lambda}, j_{\lambda}^! \phi_{\zeta *}^{-1} \nabla_{\zeta + \mu}^{\IW}).
         \end{align*}
         We can identify $\phi_{\zeta *}^{-1} \nabla_{\zeta + \mu}^{\IW}$ with the $*$-pushforward of $(\chi_{\mu}^{\zeta})^* u_! \k [\langle \zeta + \mu, 2\rho \rangle]$ along the obvious embedding $z^{-\zeta} \Gr_{\zeta + \mu}^+ \hookrightarrow \Gr$.
         As a result, by base change there is an isomorphism
         \[H_{\G_m}^{\langle \zeta - \lambda, 2\rho \rangle - n - 1} (\Gr^{\lambda}, j_{\lambda}^! \phi_{\zeta *}^{-1} \nabla_{\zeta + \mu}^{\IW}) \cong H_{\G_m}^{\langle \mu + 2\zeta - \lambda, 2\rho \rangle - n - 1} (Y_{\lambda, \mu}, h_{\lambda, \mu}^! (\chi_{\mu}^{\zeta})^* u_! \k). \]

         Note that $\chi_{\mu}^{\zeta}$ is smooth of relative dimension $\langle \zeta +\mu , 2\rho \rangle - 1$, so we have an isomorphism
         \begin{align*}
            \DD R\Gamma_{\G_m} (h_{\lambda, \mu}^! (\chi_{\mu}^{\zeta})^* u_! \k) &\cong R\Gamma_{c, \G_m} (\DD h_{\lambda, \mu}^! (\chi_{\mu}^{\zeta})^* u_! \k) \\
            &\cong R\Gamma_{c, \G_m} (h_{\lambda, \mu}^* (\chi_{\mu}^{\zeta})^! u_* \k [2]) \\
            &\cong R\Gamma_{c, \G_m} (h_{\lambda, \mu}^* (\chi_{\mu}^{\zeta})^* u_* \k [2\langle \zeta + \mu,  2\rho \rangle])
         \end{align*}
         We conclude by taking cohomology that 
         \begin{align*}
            H_{\G_m}^{\langle \mu + 2\zeta - \lambda, 2\rho \rangle - n - 1} (Y_{\lambda, \mu}, h_{\lambda, \mu}^! (\chi_{\mu}^{\zeta})^* u_! \k) &\cong H^{\langle \lambda + \mu, 2\rho \rangle + n + 1} \left( R\Gamma_{c, \G_m} ( h_{\lambda, \mu}^* (\chi_{\mu}^{\zeta})^* u_* \k) \right)^*.
         \end{align*}
    \end{proof}

    We define 
    \[Y_{\lambda, \mu}^0 = Y_{\lambda, \mu} \cap (\chi_{\mu}^{\zeta})^{-1} (0).\]
    This is a closed subvariety of $Y_{\lambda, \mu}$. We will write $s' : Y_{\lambda, \mu}^0 \hookrightarrow Y_{\lambda, \mu}$ for the inclusion map.

    \begin{lemma}\label{lem:dim_Y_lambda_nu_0}
        Let $\lambda, \mu \in \bfXv^+$. If $\lambda \neq \mu$, then $\dim Y_{\lambda, \nu}^0 < \langle \lambda + \mu, \rho \rangle$.
    \end{lemma}
    \begin{proof}
        Our proof will make use of holonomic $D$-modules. 
        To simplify notation, we will use the same notation as constructible sheaves with $\C$-coefficients.
        The key to our argument will be the geometric Casselman--Shalika formula established independently in \cite{NP} and \cite{FGV}.
        It states that
        \begin{equation}\label{eq:gcsf}
            H_c^i (S_{\mu}, \scrJ_{!*} (\lambda, \C)\vert_{S_{\mu}} \otimes^L (\chi_{\mu}')^* (\exp)) = \begin{cases}\C & \text{if }\lambda = \mu \text{ and } i = \langle 2\rho, \lambda \rangle, \\ 0 & \text{otherwise,}\end{cases}
        \end{equation}
        where $\exp$ denotes the exponential $D$-module on $\A^1$.

        By \cite[Theorem 3.2]{MV}, we have that $\dim (\Gr^{\lambda} \cap S_{\mu}) = \langle \lambda + \mu, \rho \rangle$. It follows that $\dim (Y_{\lambda, \mu}) \leq \langle \lambda + \mu, \rho \rangle$.
        If this inequality is strict, then our claim is obvious. 
        Otherwise, the lemma is equivalent to the claim that $\chi_{\mu}' \circ h_{\lambda, \mu} : Y_{\lambda, \mu} \to \G_a$ is nonzero on every irreducible component $Z$ of dimension $\langle \lambda + \mu ,\rho \rangle$ in $Y_{\lambda, \mu}$.
        Now $Z$ is open and dense in an irreducible component $Z'$ of $\Gr^{\lambda} \cap S_{\mu}$, so it suffices to show that $\chi_{\mu}'$ is nonzero on $Z'$.
        Assume that $\chi_{\mu}' (Z') = 0$. We then must have that $((\chi_{\mu}')^* \exp)\vert_{Z'} \cong \underline{\C}_{Z'}$, so there is an isomorphism
        \begin{equation}\label{eq:dim_Y_lambda_nu_0_1}
            H_c^{\langle \lambda + \mu, 2\rho \rangle} (Z', ((\chi_{\mu}')^* \exp)\vert_{Z'}) \cong H_c^{\langle \lambda + \mu, 2\rho \rangle} (Z', \k).
        \end{equation}
        Note that the left-hand side of (\ref{eq:dim_Y_lambda_nu_0_1}) is nonzero since the right-hand side is generated by the dual fundamental class of $Z'$.
        In particular, we must have that 
        \[H_c^{\langle \lambda + \mu, 2\rho \rangle} (\Gr^{\lambda} \cap S_{\mu}, ((\chi_{\mu}')^* \exp) \vert_{\Gr^{\lambda} \cap S_{\mu}}) \neq 0.\]
        The geometric Casselman--Shalika formula (\ref{eq:gcsf}) implies that $H_c^i (S_{\mu}, \scrF\vert_{S_{\mu}} \otimes^L (\chi_{\mu}')^* \exp) = 0$ for $i \neq \langle 2\rho, \mu \rangle$ and $\scrF \in P_{\L^+ G} (\Gr, \C)$.
        Therefore, the canonical morphism $j_{\lambda !} \underline{\C}_{\Gr^{\lambda}} [\langle \lambda, 2\rho \rangle]\to \scrJ_! (\lambda, \C)$ induces an isomorphism
        \[H_c^{\langle \lambda + \mu, 2\rho \rangle} (S_{\mu}, (j_{\lambda !} \underline{\C}_{\Gr^{\lambda}})\vert_{S_{\mu}} \otimes^L (\chi_{\mu}')^* \exp) \stackrel{\sim}{\to} H_c^{\langle \mu, 2\rho \rangle} (S_{\mu}, \scrJ_! (\lambda, \C)\vert_{S_{\mu}} \otimes^L (\chi_{\mu}')^* \exp). \]
        By base change, the left-hand side identifies with $H_c^{\langle \mu, 2\rho \rangle} (\Gr^{\lambda} \cap S_{\mu}, ((\chi_{\mu}')^* \exp)\vert_{\Gr^{\lambda} \cap S_{\mu}})$.
        However, since $\scrJ_! (\lambda, \C) \cong \scrJ_{!*} (\lambda, \C)$ (see for example \cite[Proposition 1]{Ga01}), the right-hand side vanishes by the geometric Casselman--Shalika formula (\ref{eq:gcsf}).
        This contradicts our assumption that $\chi_{\mu} (Z') = 0$ which completes the proof.
    \end{proof}

    \begin{lemma}\label{lem:cohom_of_Y_mu_lambda}
        Let $\lambda, \mu \in \bfXv^+$ such that $\lambda \neq \mu$, then
        \[H_c^{\langle \lambda + \mu, 2\rho \rangle + n} (Y_{\lambda, \mu}, ((\chi_{\mu}^{\zeta})^* u_* \k)\vert_{Y_{\lambda, \mu}}) = 0\]
        for all $n > 0$.    
    \end{lemma}
    \begin{proof}
        Consider the excision triangle
        \[s_*' \uk_{Y_{\lambda, \mu}^0} [-2] \to \uk_{Y_{\lambda, \mu}} \to ((\chi_{\mu}^{\zeta})^* u_* \k)\vert_{Y_{\lambda, \mu}} \to.\]
        We then get a long exact sequence
        \[\ldots \to H_c^{\langle \lambda + \mu, 2\rho \rangle + n} (Y_{\lambda, \mu}, \k) \to H_c^{\langle \lambda + \mu, 2\rho \rangle + n} (Y_{\lambda, \mu},  ((\chi_{\mu}^{\zeta})^* u_* \k)\vert_{Y_{\lambda, \mu}}) \to H_c^{\langle \lambda + \mu, 2\rho \rangle + n - 1} (Y_{\lambda, \mu}^0, \k) \to \ldots.\]
        Recall that $Y_{\lambda, \mu}$ is a subvariety of the Mirković--Vilonen cycle $\Gr^{\lambda} \cap S_{\mu}$.
        By Lemma \ref{lem:dim_Y_lambda_nu_0}, we have that $\dim Y_{\lambda, \nu}^0 < \langle \lambda + \mu, \rho \rangle$.
        Moreover, in the proof of Lemma \ref{lem:dim_Y_lambda_nu_0}, we saw that $\dim Y_{\lambda, \nu} \leq \langle \lambda + \mu, \rho \rangle$.
        Therefore, we conclude that both $H_c^{\langle \lambda + \mu, 2\rho \rangle + n} (Y_{\lambda, \mu}, \k) = 0$ and $H_c^{\langle \lambda + \mu, 2\rho \rangle + n - 1} (Y_{\lambda, \mu}^0, \k) = 0$.
        The desired vanishing then follows from exactness.
    \end{proof}

    \begin{lemma}\label{lem:hom_vanishing_between_Phi_and_costds}
        Assume that $\k$ is a field. For any $\lambda, \mu \in \bfXv^+$ with $\lambda \neq \mu$ we have
        \[\Hom_{P_{\IW, *} (\Gr, \k)} (\Phi_* (\scrJ_! (\lambda, \k)), \nabla_{\zeta + \mu}^{\IW} (\k)) = 0.\]
    \end{lemma}
    \begin{proof}
        By Lemma \ref{lem:hom_and_cohomology_std_costd}, it suffices to prove that
        \[H^{\langle \lambda + \mu, 2\rho \rangle+ 1} \left( R\Gamma_{c, \G_m} (((\chi_{\mu}^{\zeta})^* u_* \k)\vert_{Y_{\lambda, \mu}} )\right) = 0.\]
        We can compute this cohomology by taking the long exact sequence of the distinguished triangle from Lemma \ref{lem:G_m_equiv_cohomology_comparison},
        \begin{align*}
            \ldots \to H_c^{\langle \lambda + \mu, 2\rho \rangle+ 1} \left( Y_{\lambda, \mu}, (((\chi_{\mu}^{\zeta})^* u_* \k)\vert_{Y_{\lambda, \mu}} )\right) &\to H^{\langle \lambda + \mu, 2\rho \rangle+ 1} \left( R\Gamma_{c, \G_m} (((\chi_{\mu}^{\zeta})^* u_* \k)\vert_{Y_{\lambda, \mu}} )\right) \\
            &\to H^{\langle \lambda + \mu, 2\rho \rangle+ 3} \left( R\Gamma_{c, \G_m} (((\chi_{\mu}^{\zeta})^* u_* \k)\vert_{Y_{\lambda, \mu}} )\right) \to \ldots
        \end{align*}
        The first term vanishes by Lemma \ref{lem:cohom_of_Y_mu_lambda} while the last term vanishes by Lemma \ref{lem:hom_and_cohomology_std_costd}.
        The lemma then follows by exactness.
    \end{proof}

    \begin{proposition}\label{prop:phi_on_std_costds}
        For $\lambda \in \bfXv^+$, there are isomorphisms
        \[\Phi_? (\scrJ_! (\lambda, \k)) \cong \Delta_{\zeta + \lambda}^{\IW} (\k) \qquad\text{and}\qquad \Phi_? (\scrJ_* (\lambda, \k)) \cong \nabla_{\zeta + \lambda}^{\IW} (\k).\]
    \end{proposition}
    \begin{proof}
        We will first prove that $\Phi_* (\scrJ_! (\lambda, \k)) \cong \Delta_{\zeta + \lambda}^{\IW} (\k)$. 
        
        Let $\tau : \L^+ G \to \Gr$ be the projection.
        We will write $m_{\zeta, \lambda}$ for the restriction of the multiplication map $\L G \times^{\L^+ G} \Gr \to \Gr$ to $\tau^{-1} (\overline{\Gr^{\zeta}}) \times^{\L^+ G} \overline{\Gr^{\lambda}}$.
        It is explained in the proof of \cite[Theorem 3.9]{BGMRR} that the following hold:
        \begin{itemize}
            \item $m_{\zeta, \lambda}$ takes values in $\overline{\Gr^{\zeta + \lambda}} = \overline{\Gr_{\zeta + \lambda}^+}$;
            \item the restriction of $m_{\zeta, \lambda}$ to the preimage of $\Gr_{\zeta + \lambda}^+$ is an isomorphism;
            \item this preimage is contained in $\tau^{-1} (\Gr_{\zeta}^+) \times^{\L^+ G} \Gr^{\lambda}$.
        \end{itemize}
        One can then deduce that $\Phi_* (\scrJ_! (\lambda, \k))$ is supported on $\overline{\Gr_{\zeta + \lambda}^+}$ and that its restriction to $\Gr_{\zeta + \lambda}^+$ is isomorphic to $\scrK_{\lambda}^* (\k) [\langle \lambda, 2\rho \rangle]$.
        Therefore, by Lemma \ref{lem:recollement}, there is a canonical morphism
        \[f_{\lambda}^{\k} : \Delta_{\zeta + \lambda}^{\IW} (\k) \to \Phi_* (\scrJ_! (\lambda, \k))\]
        whose restriction to $\Gr_{\zeta + \lambda}^+$ is an isomorphism.

        We claim that $f_{\lambda}^{\k}$ is an isomorphism.
        Since our constructions are compatible with extension-of-scalars (see Lemma \ref{lem:eos_and_spherical} and Lemma \ref{lem:eos_commutes_with_sh_functors}), it suffices to prove that $f_{\lambda}^{\Z}$ is an isomorphism.
        If $\k \in \{\Q, \F_p\}$, by Lemma \ref{lem:hom_vanishing_between_Phi_and_costds} and Lemma \ref{lem:h_wt_str}, we know that $\Phi_* (\scrJ_! (\lambda, \k))$ has no quotient of the form $\IC_{\zeta + \mu}^{\IW} (\k)$ with $\mu \neq \lambda$; therefore, $f_{\lambda}^{\k}$ is surjective.
        Since $f_{\lambda}^{\F_p}$ is surjective for all primes $p$, we can deduce that $f_{\lambda}^{\Z}$ is also surjective.
        Note that by Lemma \ref{lem:semisimple_char_0}, $\Delta_{\zeta + \lambda}^{\IW} (\Q)$ is simple, as a result $f_{\lambda}^{\Q}$ being surjective forces $f_{\lambda}^{\Q}$ to be an isomorphism.
        In particular, $\ker (f_{\lambda}^{\Z})$ is a torsion object.
        However, $\ker (f_{\lambda}^{\Z})$ is a subobject of $\Delta_{\zeta + \lambda}^{\IW} (\Q)$ which is torsion free; therefore, $\ker (f_{\lambda}^{\Z})$ is also torsion free and hence $0$.
        We conclude that $f_{\lambda}^{\Z}$ is an isomorphism. 

        From the first case, by Lemma \ref{lem:eos_and_spherical} and Lemma \ref{lem:iw_action_commutes_with_D_and_T}, we can deduce that $\Phi_! (\scrJ_* (\lambda, \k)) \cong \nabla_{\zeta + \lambda}^{\IW} (\k)$.
        The other two isomorphism then follow from the already established isomorphisms using  Lemma \ref{lem:iw_action_commutes_with_D_and_T}, Lemma \ref{lem:translation_and_sh_functors}, and Example \ref{ex:translation_functor_for_A1}.
    \end{proof}

    \subsection{Completing the Proof}\quad

    \begin{midsecproof}{Theorem \ref{thm:cs}}
        It suffices to prove that the induced functor
        \[D^b P_{\L^+ G} (\Gr, \k) \stackrel{D^b \Phi_?}{\longrightarrow} D^b P_{\IW, ?} (\Gr, \k) \stackrel{\ref{prop:realization_is_equivalence}}{\longrightarrow} D_{\IW, ?}^b (\Gr, \k)  \]
        is an equivalence of categories. It is easy to see that $D^b P_{\L^+ G} (\Gr, \k)$ (resp. $D_{\IW, ?}^b (\Gr, \k)$) is generated as a triangulated category by both $\{\scrJ_! (\lambda, \k)\}_{\lambda \in \bfXv^+}$ and $\{\scrJ_* (\lambda, \k)\}_{\lambda \in \bfXv^+}$ (resp. $\{\Delta_{\zeta + \lambda}^{\IW}\}_{\lambda \in \bfXv^+}$ and $\{\nabla_{\zeta + \lambda}^{\IW}\}_{\lambda \in \bfXv^+}$).
        Therefore, by Proposition \ref{prop:phi_on_std_costds}, it suffices to prove that for any $\lambda, \mu \in \bfXv^+$ and any $n \in \Z$, the functor $\Phi_?$ induces an isomorphism
        \begin{equation}\label{eq:cs_1}
            \Ext_{P_{\L^+ G} (\Gr, \k)}^n (\scrJ_! (\lambda, \k), \scrJ_* (\mu, \k)) \stackrel{\sim}{\to} \Hom_{D_{\IW, ?}^b (\Gr, \k)} (\Gr, \k) (\Delta_{\zeta + \lambda}^{\IW} (\k), \nabla_{\zeta + \mu}^{\IW} (\k) [n]).
        \end{equation}
        By (\ref{eq:hom_std_costd}) and Lemma \ref{lem:hom_vanishing_for_spherical}, we have that both sides of (\ref{eq:cs_1}) vanish unless $\lambda = \mu$ and $n = 0$.
        Let $\xi^{\k} : \scrJ_! (\lambda, \k) \to \scrJ_* (\mu, \k)$ denote the canonical map induced by adjunction. 
        It suffices to show that $\Phi_? (\xi^{\F})$ is nonzero for every field $\k \to \F$. 
        Of course, the cone of $\xi^{\F}$ is supported on $\overline{\Gr^{\lambda}} \setminus \Gr^{\lambda}$.
        On the other hand, by Proposition \ref{prop:phi_on_std_costds}, the cone of $\Phi_? (\xi^{\F})$ is supported on $\overline{\Gr_{\zeta + \lambda}^+} \setminus \Gr_{\zeta + \lambda}^+$.
        In particular, we must have that $\Phi_? (\xi^{\F})$ is nonzero.
        Therefore, (\ref{eq:cs_1}) is an isomorphism for $\lambda = \mu$ and $n = 0$ which completes the proof.
    \end{midsecproof}


\begin{thebibliography}{S}

        \bibitem[Ach]{Ac}
        P. Achar, {\it Perverse sheaves and applications to representation theory}, Mathematical Surveys and Monographs, 258, Amer. Math. Soc., Providence, RI, [2021] \copyright 2021.

        \bibitem[AB]{AB}
        S. Arkhipov and R. Bezrukavnikov, {\it Perverse sheaves on affine flags and Langlands dual group}, Israel J. Math. {\bf 170} (2009), 135--183.

        \bibitem[ABBGM]{ABBGM}
        S. Arkhipov, R. Bezrukavnikov, A. Braverman, D. Gaitsgory, and I. Mirković., {\it Modules over the small quantum group and semi-infinite flag manifold}, Transform. Groups {\bf 10} (2005), no.~3-4, 279--362.

        \bibitem[BBDG]{BBD}
        A. Beĭlinson, J. Berstein, P. Deligne, and O. Gabber, {\it Faisceaux pervers}, Ast\'erisque {\bf 2018}, no.~100, vi+180 pp.

        \bibitem[Bez1]{Bez01}
        R. Bezrukavnikov, {\it Quasi-exceptional sets and equivariant coherent sheaves on the nilpotent cone}, Represent. Theory {\bf 7} (2003), 1--18.

        \bibitem[BF]{BF}
        R. Bezrukavnikov and M. Finkelberg, {\it Equivariant Satake category and Kostant--Whittaker reduction}, Mosc. Math. J. {\bf 8} (2008), no.~1, 39--72, 183.

        \bibitem[BGMRR]{BGMRR}
        R. Bezrukavnikov, D. Gaitsgory, I. Mirković, L. Rider, and S. Riche, {\it An Iwahori--Whittaker model for the Satake category}, J. \'Ec. polytech. Math. {\bf 6} (2019), 707--735.

        \bibitem[BGS]{BGS}
        A.~A. Beilinson, V. Ginsburg and W. Soergel, {\it Koszul duality patterns in representation theory}, J. Amer. Math. Soc. {\bf 9} (1996), no.~2, 473--527.

        \bibitem[CvdHS]{CvdHS}
        R. Cass, T. van den Hove, and J. Scholbach, {\it Exponential Motives on the Affine Grassmannian}, preprint \href{https://arxiv.org/abs/2603.23435}{arXiv:2603.23435}.

        \bibitem[FGV]{FGV}
        E.~V. Frenkel, D. Gaitsgory and K. Vilonen, {\it Whittaker patterns in the geometry of moduli spaces of bundles on curves}, Ann. of Math. (2) {\bf 153} (2001), no.~3, 699--748.

        \bibitem[Gai1]{Ga01} 
        D. Gaitsgory, {\it Construction of central elements in the affine Hecke algebra via nearby cycles}, Invent. Math. {\bf 144} (2001), no.~2, 253--280.

        \bibitem[Gai2]{GaiWhit}
        D. Gaitsgory, {\it The local and global versions of the Whittaker category}, Pure Appl. Math. Q. {\bf 16} (2020), no.~3, 775--904. 

        \bibitem[GL]{GL} 
        D. Gaitsgory and S. Lysenko, {\it Metaplectic Whittaker category and quantum groups: the ``small'' FLE}, preprint \href{https://arxiv.org/abs/1903.02279}{arXiv:1903.02279}.

        \bibitem[JMW]{JMW}
        D. Juteau, C. Mautner and G. Williamson, {\it Parity sheaves}, J. Amer. Math. Soc. {\bf 27} (2014), no.~4, 1169--1212.

        \bibitem[Lau]{Lau}
        G. Laumon, {\it Transformation de Fourier homog\`ene}, Bull. Soc. Math. France {\bf 131} (2003), no.~4, 527--551.

        \bibitem[Lus]{Lu1}
        G. Lusztig, {\it Singularities, character formulas, and a $q$-analog of weight multiplicities}, in {\it Analysis and topology on singular spaces, II, III (Luminy, 1981)}, 208--229, Ast\'erisque, 101-102, Soc. Math. France, Paris.

        \bibitem[MV]{MV}
        I. Mirković and K. Vilonen, {\it Geometric Langlands duality and representations of algebraic groups over commutative rings}, Ann. of Math. (2) {\bf 166} (2007), no.~1, 95--143.

        \bibitem[NP]{NP}
        B.~C. Ng\^o{} and P. Polo, {\it R\'esolutions de Demazure affines et formule de Casselman-Shalika g\'eom\'etrique}, J. Algebraic Geom. {\bf 10} (2001), no.~3, 515--547.

         \bibitem[RSW]{RSW}
         S. Riche, W. Soergel and G. Williamson, {\it Modular Koszul duality}, Compos. Math. {\bf 150} (2014), no.~2, 273--332.

        \bibitem[San1]{Sa1}
        C. Sandvik, {\it Endoscopy for modular Hecke categories}, preprint \href{https://arxiv.org/abs/2508.10214}{arXiv:2508.10214}.

    \end{thebibliography}
\end{document}